\documentclass{amsart}
\usepackage{amsmath, amsthm, amssymb}
\usepackage{graphics}
\usepackage{pinlabel}

\usepackage{color}

\DeclareMathOperator{\Mod}{MCG}
\DeclareMathOperator{\PMod}{PMCG}
\DeclareMathOperator{\Cantor}{Cantor}

\begin{document}

\newtheorem{theorem}{Theorem}[subsection]
\newtheorem{lemma}[theorem]{Lemma}
\newtheorem{corollary}[theorem]{Corollary}
\newtheorem{conjecture}[theorem]{Conjecture}
\newtheorem{proposition}[theorem]{Proposition}
\newtheorem*{proposition*}{Proposition}
\newtheorem*{question}{Question}
\newtheorem*{answer}{Answer}
\newtheorem{problem}[theorem]{Problem}
\newtheorem*{simplex_theorem}{Superideal Simplex Theorem}
\newtheorem*{claim}{Claim}
\newtheorem*{criterion}{Criterion}
\newtheorem*{theorem*}{Theorem}
\newtheorem*{lemma*}{Lemma}
\newtheorem*{corollary*}{Corollary}
\theoremstyle{definition}
\newtheorem{definition}[theorem]{Definition}
\newtheorem{construction}[theorem]{Construction}
\newtheorem{notation}[theorem]{Notation}
\newtheorem{convention}[theorem]{Convention}
\newtheorem*{warning}{Warning}
\newtheorem*{assumption}{Simplifying Assumptions}

\theoremstyle{remark}
\newtheorem{remark}[theorem]{Remark}
\newtheorem{example}[theorem]{Example}
\newtheorem{scholium}[theorem]{Scholium}
\newtheorem*{case}{Case}

\def\Id{\text{Id}}
\def\H{\mathbb H}
\def\Z{\mathbb Z}
\def\N{\mathbb N}
\def\R{\mathbb R}
\def\C{\mathbb C}
\def\CP{{\mathbb {CP}}}
\def\CC{\mathcal C}
\def\HC{\mathcal H}
\def\S{\mathcal S}
\def\P{\mathcal P}
\def\Sph{\mathbb S}
\def\P{\mathcal P}
\def\Q{\mathbb Q}
\def\L{L}
\def\A{\mathcal A}
\def\E{\mathcal E}
\def\U{\mathcal U}
\def\W{\mathcal W}
\def\I{\mathbb I}
\def\Homeo{\textnormal{Homeo}}
\def\inte{\textnormal{int}}
\def\scl{\textnormal{scl}}
\def\loc{\textnormal{loc}}
\def\RG{{\mathcal{R}'}}
\def\RGC{\mathcal{R}}

\def\Or{{\mathcal{O}}}

\newcommand{\marginal}[1]{\marginpar{\tiny #1}}

\title[Two simultaneous actions of big mapping class groups]{Two simultaneous actions\\
of big mapping class groups}
\author{Juliette Bavard}
\address{Univ Rennes, CNRS, IRMAR - UMR 6625, F-35000 Rennes, France}
\email{juliette.bavard@univ-rennes1.fr}
\author{Alden Walker}
\address{Center for Communications Research \\ La Jolla, CA 92121 \\ USA}
\email{akwalke@ccrwest.org}

\begin{abstract} We study two actions of big mapping class groups.
The first is an action by isometries on a Gromov-hyperbolic graph. The second is an action by homeomorphisms on a circle in which the vertices of the graph naturally embed. 

The first two parts of the paper are devoted to the definition of objects and tools needed to introduce these two actions; in particular, we define and prove the existence of \emph{equators} for infinite type surfaces, we define the hyperbolic graph and the circle needed for the actions, and we describe the Gromov-boundary of the graph using the embedding of its vertices in the circle. 

The third part focuses on some fruitful relations between the dynamics of the two actions. For example, we prove that loxodromic elements (for the first action) necessarily have rational rotation number (for the second action). In addition, we are able to construct non trivial quasimorphisms on many subgroups of big mapping class groups, even if they are not acylindrically hyperbolic. 

\end{abstract}

\maketitle

\setcounter{tocdepth}{1}
\tableofcontents

\section{Introduction}

\subsection{Context}  
\subsubsection{Dynamics} The study of mapping class groups of infinite type surfaces is motivated by various dynamical questions related to group actions on surfaces, one-variable complex dynamics, foliations of $3$-manifolds, subgroups of $\Homeo^+(\Cantor)$, etc (more details on these motivations are given by D. Calegari in~\cite{Calegari-blog}). These ``big mapping class groups'' are uncountable and thus very different from the mapping class groups of finite type surfaces. However, a good understanding of these groups and their subgroups would be useful to approach some basic questions for which only a few techniques exist. Here is an example of such a question from K. Mann:
\begin{question}
Does there exist a finitely generated group which is torsion free and which does not act on the plane by homeomorphisms?
\end{question}
As far as we know, no example of such a group is known: this illustrates how difficult it is to construct obstructions to such actions. Nevertheless, infinite type surfaces appear when we remove certain orbits of a group action on a finite type surface. The associated big mapping class groups are a possible fruitful approach to the search of obstructions to the action. A concrete example can be found in \cite{Calegari-circular}, where D. Calegari studies mapping class groups of complements of compact, totally disconnected subsets of the plane. He uses these groups to prove in particular that any group of orientation preserving $\CC^1$-diffeomorphisms of the plane with a bounded orbit is circularly orderable.

In this setting, it becomes very useful to know which subgroups of big mapping class groups admit non trivial quasimorphisms. Indeed, the existence of such quasimorphisms gives specific obstructions on some group actions on surfaces. The space of non trivial quasimorphisms $\tilde Q(G)$ of a group $G$ is a group invariant which pulls back through surjections: if there exists a surjective homomorphism from a group $G$ to a group $H$, then $\tilde Q(G)$ contains a copy of $\tilde Q(H)$. Calegari's construction gives us surjective homomorphisms from groups acting on surfaces to subgroups of (big) mapping class groups. It follows that the existence of quasimorphisms on subgroups of big mapping class groups is an obstruction to some actions of groups without quasimorphisms.

\subsubsection{Gromov-hyperbolic spaces}

Gromov-hyperbolic graphs associated to a surface $S$ are useful tools to study the mapping class group $\Mod(S)$ of~$S$ and its subgroups: an action by isometries on such a graph gives access to the machinery of geometric group theory to determine properties of the group. For finite topological surfaces, i.e. compact surfaces possibly with finitely many punctures, the curve graph and the arc graph of the surface are well studied Gromov-hyperbolic spaces which give information about $\Mod(S)$ and its subgroups. However, these curve and arc graphs have finite diameter for infinite type surfaces. The ray graph of the plane minus a Cantor set was originally defined by D. Calegari in~\cite{Calegari-blog}, and the first author proved in \cite{Juliette} that it is Gromov-hyperbolic and has infinite diameter. Gromov-hyperbolic graphs related to infinite type surfaces have then also been studied by J. Aramayona, A. Fossas and H. Parlier in \cite{Aramayona-F-P}, by J. Aramayona and F. Valdez in~\cite{Aramayona-V}, by M. Durham, F. Fanoni and N. Vlamis in~\cite{Durham-F-V}, and by A. Rasmussen in  \cite{Rasmussen}. 

\subsubsection{Quasimorphisms}
While it is a first step to have a Gromov-hyperbolic graph on which a group $G$ acts with loxodromic elements, this alone is not enough to prove the existence of non trivial quasimorphisms on $G$. In \cite{Bestvina-Fujiwara}, M. Bestvina and K. Fujiwara give an additional criterion which certifies the existence of non trivial quasimorphisms on $G$: they prove that whenever $G$ has a \emph{weakly properly discontinuous} (WPD) action on the graph, then the space of non trivial quasimorphisms on $G$ is infinite dimensional. However, the first author and A. Genevois proved in \cite{Juliette-Anthony} that big mapping class groups are not acylindrically hyperbolic, and thus have no WPD action on any graph. Hence we need new techniques to prove that the big mapping class groups under consideration (and some of their subgroups) admit non trivial quasimorphisms. We will use a weaker criterion, given again by Bestvina-Fujiwara in \cite{Bestvina-Fujiwara}, which guarantees the existence of non trivial quasimorphisms whenever the action of $G$ has two \emph{"anti-aligned"} loxodromic elements: roughly speaking, two loxodromic elements $g$ and $h$ are \emph{anti-aligned} if no element of $G$ maps long segments of the quasi-axis of $g$ to a close neighborhood of a quasi-axis of $h$. We will prove the existence of such elements in big mapping class groups in order to get non trivial quasimorphisms. Note that our methods are very different from the one given in \cite{Juliette}, were we used the description of the neighborhood of one particular axis on a particular loxodromic element of $\Mod(\R^2 -\Cantor)$ to prove the \emph{anti-aligned criterion}: here we use together a description of the Gromov-boundary of the graph and the action by homeomorphisms on the circle, and we are able to construct quasimorphisms on many more groups.

\subsection{Results} 

Our main goal in this paper is to generalize and extend
the results of~\cite{Juliette} and~\cite{boundary} to
the case of more general infinite type surfaces.
That is, to understand the structure of
mapping class groups of infinite type surfaces which fix an isolated marked puncture.
If $S$ is a surface (we define \emph{surface} below and in
Section~\ref{sec:surfaces}), and $p\in S$ is a marked puncture,
then we define $\Mod(S;p)$ to be the group of homeomorphisms
of $S$ which fix $p$, taken up to isotopy.

\subsubsection{Equators for infinite type surfaces}

One of the most important tools in~\cite{boundary} was
an obvious equator in $S^2 - \Cantor$: a locally finite
family of mutually disjoint and proper simple arcs whose complement
is two topological disks.  The equator is essentially
equivalent to a nice fundamental domain.
For general surfaces, the existence of an equator is nontrivial,
and is our first result.  We remark that even the decision
about what we mean by a surface and how we
represent them up to homeomorphism are not obvious.
However, fortunately we can rely on the background work
in~\cite{Richards}, which classifies our surfaces of interest
up to homeomorphism, where surface means an triangulable
$2$-manifold without boundary (but with punctures), and
we will only be interested in the cases of infinite type or
negative Euler characteristic.
Our first result produces a nice fundamental domain,
and we use its boundary as our equator:

\begin{theorem*}{\textbf{\emph{\ref{thm:hyperbolic}}}}
Let $S$ be a surface with at least one end which is not a disk or
annulus.  Then $S$ admits a complete hyperbolic metric of the first
kind, and there is a fundamental polygon for $S$ which is geodesic
and has no vertex in the interior of the hyperbolic plane.
Further, every isolated puncture is a vertex on this fundamental polygon,
and for any particular isolated puncture, the polygon can be chosen so that
the Fuchsian group for the polygon has a parabolic element fixing
the puncture.
\end{theorem*}

Along the way to proving Theorem~\ref{thm:hyperbolic}, we
give a concrete, useful parameterization of any surface and show
how to cut it into two topological disks with a carefully chosen
collection of proper simple arcs.

\subsubsection{Definition of the two actions}

The second part of the paper deals with Gromov-hyperbolic graphs of surfaces of finite or infinite type and with an isolated marked puncture and their Gromov boundaries. In~\cite{boundary}, we gave a concrete description of the Gromov boundary of the loop graph of the plane minus a Cantor set.  Namely, we showed it is homeomorphic, with suitable topologies, to the set of cliques in a graph containing the loop and ray graphs.

In this paper, we generalize results of~\cite{boundary} to
our more general class of surfaces.  In particular, we define
the loop graph $L(S;p)$ of a surface $S$ based at a puncture $p$
and show it is hyperbolic and that $\Mod(S;p)$ acts on it
by isometries.  This is the first action.

To better understand the loop graph, we define a specific cover of the surface which focuses on the distinguished puncture:  the \emph{conical cover}, which is the cover we obtain by quotienting the universal
cover by the parabolic about the puncture. The boundary of this cover is a circle $S^1$, in which the loop graph naturally embeds. Moreover, the mapping class group of the surface naturally acts by homeomorphisms on this circle.  This is the second
action.

In addition to loops on the surface, we also study geodesic rays; these
are easy to understand from the perspective of the conical cover: a
ray is the projection of the geodesic ray from $p$ to a point on the
boundary $S^1$ of the conical cover, where this ray is required to
be simple.  We can define a more general graph $\RGC$ which contains
the loops and rays, and on which $\Mod(S;p)$ also acts by isometries.
Some rays are special: they are \emph{high-filling}, which means
they are not in the connected component of any loop.
We prove results analogous to~\cite{boundary}; namely, that
these long rays appear in cliques, and the cliques
are in bijection with the Gromov boundary of the loop graph
(Theorem~\ref{theorem:boundary_bijection}).  Moreover, 
with the topology that these rays inherit from (a quotient of)
the circle, they are homeomorphic to the Gromov boundary of the
loop graph (Theorem~\ref{theorem:boundary_homeo}).  We omit
the precise statements of the theorems because they involve
a surplus of notation.

Every loxodromic element $h \in \Mod(S;p)$ has two fixed
points on the boundary of $L(S;p)$; with the above correspondence,
we get two preserved cliques of high-filling rays. When $S$
is a finite type surface and $h$ is a pseudo-Anosov, then $h$
has a loxodromic action on $\L(S;p)$ and its associated
cliques are the set of leaves of the attracting and
repelling foliations preserved by $h$ that end at the
prong singularity $p$. Thus, in this case, we can construct
the two usual preserved laminations associated to $h$ by
taking the closure of our cliques. However, loxodromic elements on
general surfaces do not necessarily preserve any finite type 
subsurface (see for example the element $h$ constructed in
Section $4.1$ of \cite{Juliette}), so minimal filling laminations
of finite type subsurfaces are not enough to describe the
whole boundary of loop graphs of infinite type surfaces.

\subsubsection{Simultaneous dynamics of the actions and applications}

In the third part of the paper, we study links between the two
actions of $\Mod(S;p)$. On one hand, we can associate to
every loxodromic element two cliques of high-filling rays,
each of them representing a point of the Gromov-boundary of the
loop graph. On the other hand, every element also acts as a
homeomorphism of the circle (seen as the boundary of the conical cover).
Using this action on the circle, we prove the following result:

\begin{theorem*}{\textbf{\emph{\ref{theorem:finite cliques}.}}} Let $S$ be a surface with an isolated marked puncture $p$.
Let $h \in \Mod(S;p)$ be a loxodromic element. The cliques $\CC^-(h)$ and $\CC^+(h)$ of high-filling rays associated to $h$ are finite and have the same number of elements.
Moreover, if we identify $S^1$ as the boundary of the conical cover,
then there exists $k\in \N$ such that the action of $h^k$ on $S^1$
has a Morse-Smale dynamics with fixed points exactly the
high-filling rays of the attractive and repulsive cliques.
\end{theorem*}

See Figure~\ref{figu:action_boundary_intro} for the picture
of the two actions in Theorem~\ref{theorem:finite cliques}.

\begin{figure}[htb]
\labellist
\pinlabel $h$ at 160 210
\pinlabel $\textrm{Loop graph }\L(S;p)$ at 200 -40
\pinlabel $\textrm{Boundary of the conical cover}$ at 750 -40
\endlabellist
\centering
\includegraphics[scale=0.3]{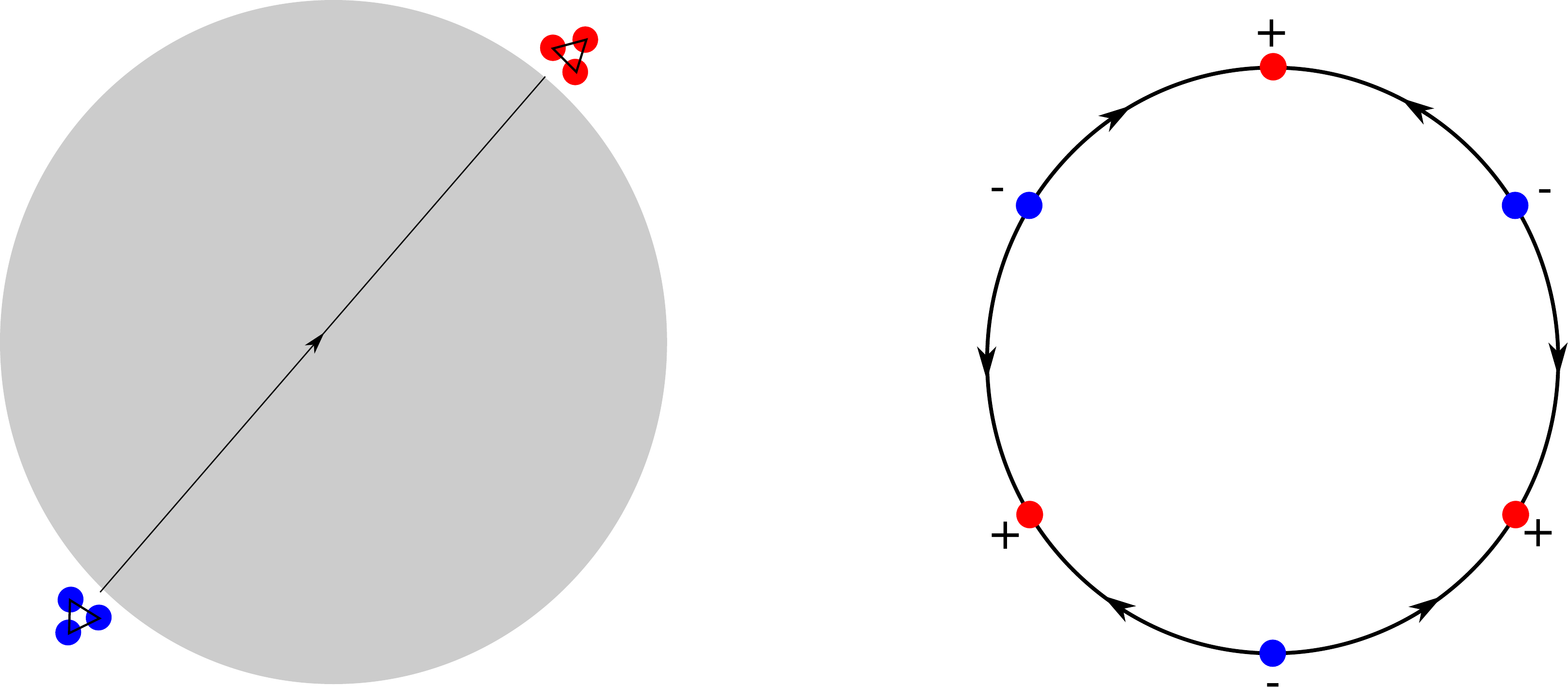}
\vspace{1cm}
\caption{On the left, the action of a loxodromic element $h$ on the loop graph: the two fixed points on the boundary are cliques. On the right, the action of $h^n$ (where $n$ is the cardinal of the associated cliques) on the boundary of the conical cover: the elements of the two attractive and repulsive cliques associated to $h$ gives the attractive and repulsive points for this circular action.}
\label{figu:action_boundary_intro}
\end{figure}

In particular, as homeomorphisms of the circle, loxodromic elements have a rational rotation number.
After Theorem~\ref{theorem:finite cliques}, we can define the
\emph{weight} of a loxodromic element as the cardinality of
each of its associated cliques. For example, the weight of the
element in the cartoon of Figure~\ref{figu:action_boundary_intro}
is $3$. We show that on an infinite type surface, any weight
is possible:

\begin{theorem*}{\textbf{\emph{\ref{theo:weights}.}}}
Let $S$ be an infinite type surface with an isolated marked puncture $p$. For any $n\in \N$,
there exist elements of $\Mod(S;p)$ with a loxodromic action on the loop graph
$\L(S;p)$ whose weight is equal to $n$.
\end{theorem*}

The proof use the existence of quasi-isometric embeddings between loop graphs of subsurfaces: this allows us to use results from finite type surfaces theory to construct our examples.  We remark that we
would like more examples of high weight loxodromic actions
which are fundamentally infinite-type, but these are difficult
to understand.  Note that the infinite-type result in
Theorem~\ref{theo:weights} is in contrast
to the situation for finite type surfaces:
\begin{lemma*}{\textbf{\emph{\ref{lemma:finite_type}.}}}
Let $S$ be a finite type surface with a marked puncture $p$.
There is a uniform bound on the weight of any element in $\Mod(S;p)$.
\end{lemma*}

We next focus on the construction of quasimorphisms for subgroups of $\Mod(S;p)$. Again using the properties of the
two simultaneous actions, we prove the following result:

\begin{theorem*}{\textbf{\emph{\ref{theo:quasimorphisms_weight}.}}} Let $S$ be a surface with an isolated marked puncture $p$.
Let $G$ be a subgroup of $\Mod(S;p)$ such that there exist two
loxodromic elements $g,h \in G$ with weights $w(g) \neq w(h)$.
Then the space $\tilde Q(G)$ of non-trivial quasimorphisms
on $G$ is infinite dimensional. 
\end{theorem*}

This theorem has several corollaries; for example, together
with Theorem~\ref{theo:weights}, it proves that 
$\tilde Q(\Mod(S;p))$ is infinite dimensional.  It also provides
a tangible criterion for understanding its subgroups.

\subsection{Organization}
As intimated in the preceding description of our
results, this paper roughly divides into three parts:

\begin{enumerate}
\item Part 1 is only about surfaces and uses
mainly topological arguments.
\item Part 2 extends the results from~\cite{Juliette,boundary}
about mapping class groups and uses a variety of topology and
geometric group theory.
\item Part 3 is about dynamics of mapping class groups and
combines parts (1) and (2) along with dynamical arguments.
\end{enumerate}

With the hope of clarifying the structure, we have explicitly
recorded the parts in the text.  Part 2 is devoted mostly to
directly generalizing the results from \cite{boundary}.
The nature of the arguments in~\cite{boundary}
means that many of the proofs generalize verbatim.
In the interest of brevity, we rely heavily on~\cite{boundary}
and merely indicate which parts require more care in
our more general situation.

\subsection{Acknowledgment}
For helpful discussions around different objects of this paper, we thank Javier Aramayona, Danny Calegari, Serge Cantat, Fran{\c c}oise Dal'bo, Steven Frankel, S\'ebastien Gou\"ezel, Jeremy Kahn, Fran\c cois Laudenbach, Fr\'ed\'eric Le Roux, Yair Minsky, Alexander Rasmussen, Juan Souto and Ferr\'an Valdez.

The first author acknowledges support from the Centre Henri Lebesgue ANR-11-LABX-0020-01.

\part{Topology and geometry of infinite type surfaces}
\section{Surfaces and embeddings}
\label{sec:surfaces}

Although the generalization of the definitions and
theorems in~\cite{boundary} is intuitively straightforward,
one of the main issues is clarifying exactly what topological
assumptions are necessary.  We could provide the
generalization in the case of finite type surfaces
without much trouble, but it seems a shame to disregard
the most general possible case in a paper specifically
intended to provide a useful theoretical basis for further
study of potentially infinite type surfaces.

\subsection{Pants}

Recall that a pair of pants is a sphere minus three open
disks, as shown in Figure~\ref{figu:3_pants_1}, right.
That is, a genus zero surface with three boundary components.
It will be convenient for us to allow up to two of the
boundary components of a pair of pants to be degenerate, meaning
up to two of the boundary components may be replaced by punctures.
Figure~\ref{figu:3_pants_1} shows the possibilities.

\begin{figure}[htb]
\labellist
\endlabellist
\centering
\vspace{0.2cm}
\includegraphics[scale=0.2]{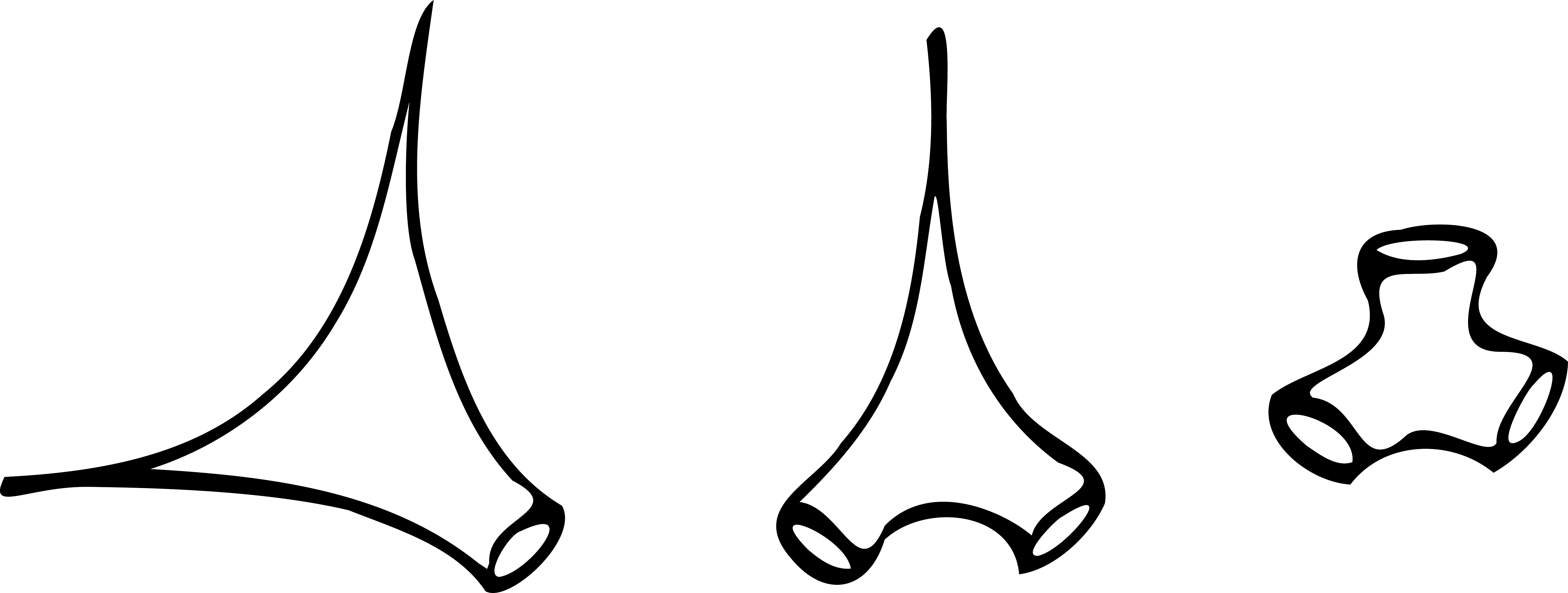}
\caption{The three allowed pairs of pants.}
\label{figu:3_pants_1}
\end{figure}

Later, we will put a standard hyperbolic metric on all
our pairs of pants.  For now, however, it is enough
to deal with them topologically.  Each pair of pants
can be cut into two disks by slicing with the plane parallel to
the page in Figure~\ref{figu:3_pants_1}.  The curves along
which the pants are sliced are called the \emph{seams}
of the pair of pants.  The result of slicing is better thought
of as two (possibly degenerate) hexagons whose sides
alternate between punctures or boundaries of the pants and seams.

If we are given a $2$-manifold $S$ and a collection of
sub-$2$-manifolds
$P$ which are each homeomorphic to a pair of pants, we say
that $P$ is a decomposition of $S$ into pairs of pants
if $\cup P = S$ and the intersection of any pair of elements
of $P$ is either empty or a boundary component of both.
It is a standard fact that any orientable surface of finite or infinite
type admits a pants decomposition.  In the next section,
we provide a specific pants decomposition for a wide class 
of $2$-manifolds, and which will be useful for our purpose.

\subsection{Infinite type surfaces}
\label{sec:infinite_type}

To start our study, we must first wrangle
an arbitrary topological surface into a useful standard
form, from which we will derive a pants decomposition
and hyperbolic metric.
By a \emph{surface}, we mean a triangulable (and thus
separable; see~\cite{Richards}) orientable 2-manifold
without boundary (but with any punctures) with either
infinite type or with negative Euler characteristic.
In one situation in Section~\ref{section:action_lox},
we will require Euler characteristic less than $-2$,
so we simply make that assumption everywhere for simplicity
(our primary interest is infinite type, so any such restriction
is essentially an unimportant ``boundary condition'').
Restricting to the case without boundary makes it possible
to use results from~\cite{Richards} and deal with
the hyperbolic metric and technicalities.  However,
we remark that the loop graph of a surface is the same
as the loop graph of that surface without its boundary
(i.e. with the boundary curves replaced by punctures).
Therefore, all the results in this paper apply equally well
to surfaces with boundary.  It is simply technically easier to deal
with the surface with the boundaries replaced by punctures.

The hard work of putting such a surface into a nice form
is accomplished
by~\cite{Richards}, and in particular~\cite{Richards}, Theorem~1:

\begin{theorem*}[\cite{Richards}, Theorem~1]
Let $S$ and $S'$ be two surfaces of the same (possibly infinite)
genus and orientability class.  Then $S$ and $S'$ are
homeomorphic if and only if their ideal boundaries $B(S)$ and
$B(S')$
are topologically equivalent (as triples of spaces).
\end{theorem*}

We would now refer to the ``ideal boundary'' of this theorem 
as the space of ends of the surface, although we temporarily use
the language of~\cite{Richards} to explain the theorem.
The ideal boundary can be thought of
as a triple by considering $B(S) \supset B'(S) \supset B''(S)$,
where $B(S)$ is the entire boundary, $B'(S)$ is the nonplanar
part (the ends accumulated by genus), and $B''(S)$ is the
nonorientable part.  This last part is not an issue for us
because we consider only orientable surfaces.

Following~\cite{Richards}, Theorem~1, we will adapt
the ideal boundary triple to be applicable to our case.
We parameterize all surfaces by the triple
$(g,K,K')$, where $g \in \bar{\N}$ is the
(possibly infinite) genus, $K$ is the set of ends,
and $K'$ is the set of ends accumulated by genus.
Then~\cite{Richards}, Theorem~1 implies that any two surfaces
with the same associated triples are homeomorphic.

In~\cite{Richards}, a second theorem constructs a surface
with any given triple in a simple way.  For our purposes,
it will be more convenient to use a different construction,
as follows.

\subsection{Core trees and surface embeddings}

It will be useful to construct a combinatorial object,
the \emph{core tree}, which will record all the information
necessary to reconstruct a surface (and some additional
information we will use later).  A core tree is a possibly
infinite tree with some vertices marked and which is allowed
the following kinds of vertices:
\begin{enumerate}
\item Marked $1$-, $2$-, and $3$-valent vertices
\item Unmarked $1$- and $3$-valent vertices
\end{enumerate}
Given a core tree $T$, we construct a topological surface $S(T)$
from $T$ by replacing each vertex with a surface which has
one boundary for each incident edge
and then gluing all pairs of boundaries indicated by the edges.
We replace:
\begin{enumerate}
\item Marked $1$-, $2$-, and $3$-valent vertices with genus one
surfaces with $1$,$2$, and $3$ boundaries, respectively.
\item Unmarked $3$-valent vertices with pairs of pants.
\item Unmarked $1$-valent vertices with nothing.
\end{enumerate}

\begin{figure}[htb]
\labellist
\endlabellist
\centering
\includegraphics[scale=0.2]{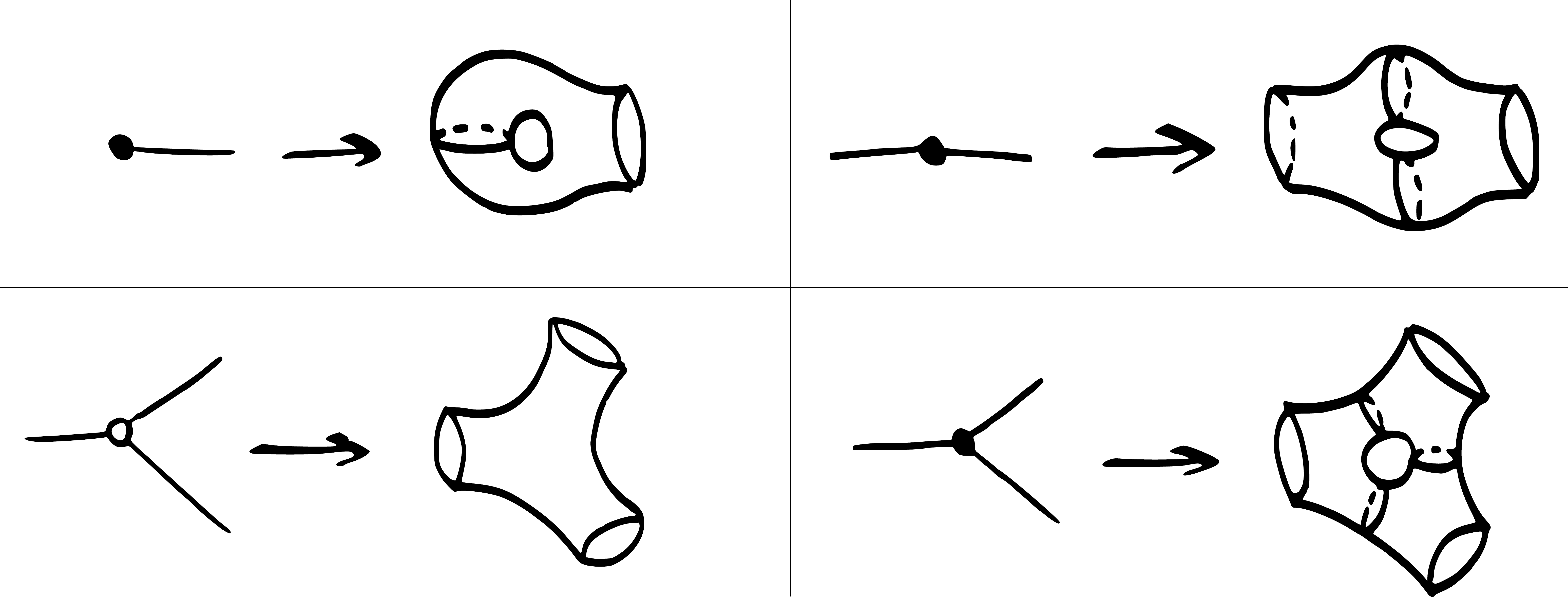}
\caption{Expanding the vertices of $T$ into the
components of $S(T)$.}
\label{figu:tree_vertices}
\end{figure}

\begin{figure}[htb]
\labellist
\endlabellist
\centering
\includegraphics[scale=0.18]{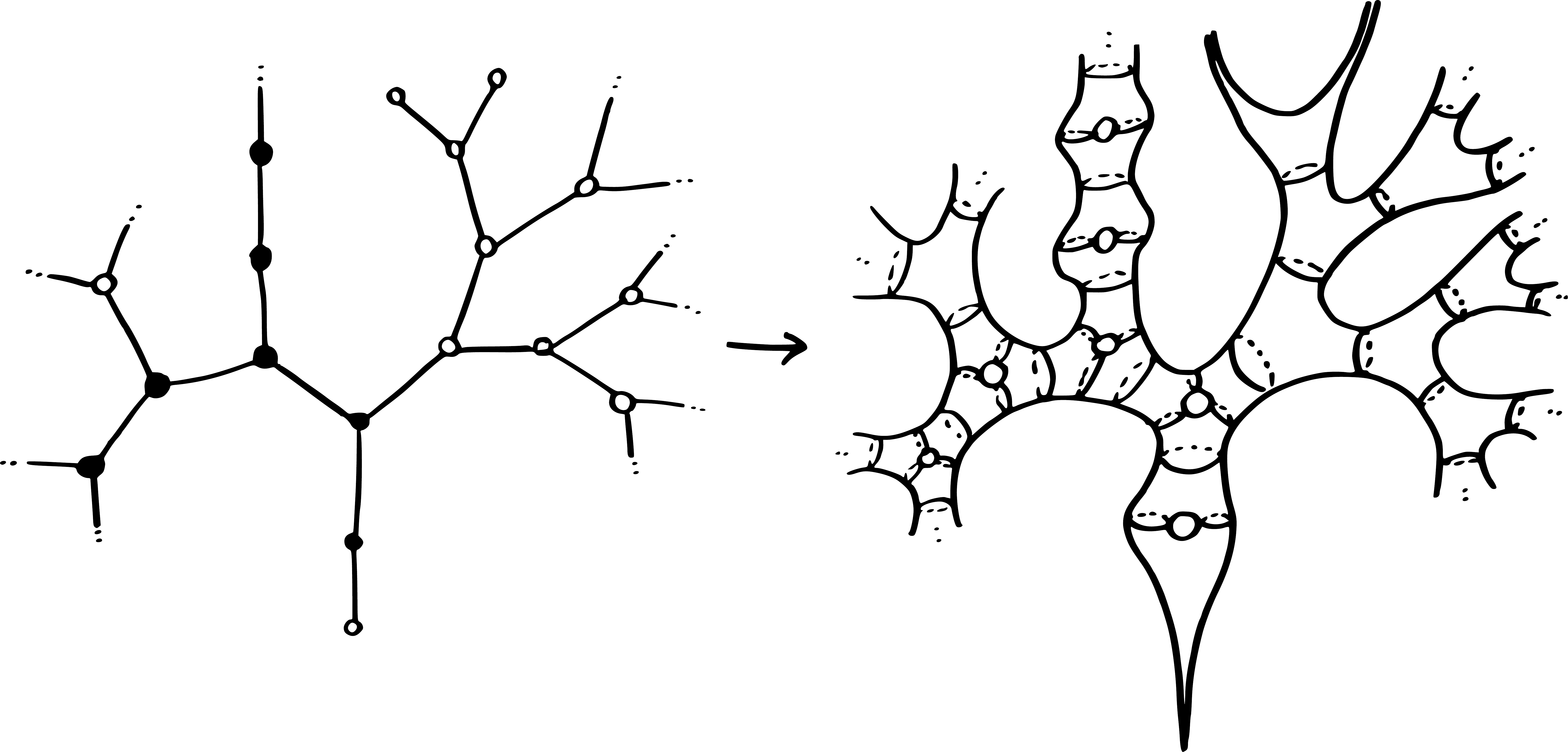}
\caption{An example core graph $T$ and the surface $S(T)$.}
\label{figu:tree_example}
\end{figure}

See Figures~\ref{figu:tree_vertices} and~\ref{figu:tree_example}
for pictures of the construction of $S(T)$.
Note that because the $1$-valent vertices are replaced with nothing,
the surface $S(T)$ may have boundary.  For technical reasons,
it will be more convenient to have these boundaries be punctures
instead, so in the construction of $S(T)$, if a boundary would
remain unglued, it is replaced with a puncture.

The core tree $T$ records the topological type of the surface $S(T)$,
but we will use it to record more.  A \emph{rooted} core tree is a
core tree with one of the vertices chosen to be a root.  This vertex
must be marked or be a $3$-valent unmarked vertex.

\begin{lemma}\label{lem:core_tree}
Given a surface $S$ which is not a sphere, torus, disk, or annulus
(i.e. $S$ has $\chi(S) < 0$ or is of infinite type),
there is a
rooted core tree $T$ so that $S \cong S(T)$.
\end{lemma}
\begin{proof}
The outline of the proof is to produce a rooted core tree $T$ and then
apply~\cite{Richards}, Theorem~1, as discussed above, to show
that $S \cong S(T)$.

First, we note that the set of ends of any surface
is a closed subset of the Cantor set (see~\cite{Richards},
Proposition~3), so given the triple $(g,K,K')$ for a surface,
$K$ is a closed subset of the standard Cantor set, and
$K'$ is a closed subset of $K$.  We take $(g,K,K')$ to be
the triple associated with our target surface $S$.

There are some special cases to be dealt with first: if $S$ has no
end and $g$ is finite, then $S$ is a closed surface, and by assumption
it has genus at least $2$.  In this case, we can take $T$ to simply
be a line of $g$ vertices, all of which are marked, and we can choose
one of them arbitrarily to be the root.  Then $S(T)$ is obviously
a closed genus $g$ surface, and we are done.

Otherwise, take $T$ to be an infinite valence $3$ tree.  We will repeatedly
modify $T$ until it becomes our desired rooted core tree.  The
ends of $T$ is a Cantor set, so we can identify $K$ with a subset
of the ends of $T$; abusing notation, we will freely think of
$K$ as the ends of $S$ or as a subset of the ends of $T$. 

If $K$ contains at least three points, then choose any three of them;
there is a unique geodesic triangle in $T$ whose vertices are these
three ends in $K$, and this triangle has a single central
vertex.  Let this vertex be the root.  If $K$ contains two
points, mark any vertex on the geodesic between these ends and let it
be the root.  If $K$ has a single point, choose any vertex in $T$,
mark it, and let it be the root.

Denote this chosen root by $r$, and consider the subtree of $T$
which is the union of all vertices and edges on any geodesic ray
from $r$ to some end in $K$.  We now rename $T$ to be this subtree,
as we imagine removing from $T$ everything not on one of these rays.
Note that now $K$ is exactly the ends of $T$.

We will now mark vertices in $T$ to get the appropriate genus.
If $g$ is finite, then $K'$ must be empty, and we mark any $g$ vertices
in $T$.  If the root $r$ is already marked, then mark $g-1$ vertices.
Note that if the root is marked, then $g\ge 1$ because a genus zero
surface with one or two ends is a disk or annulus, which we have ruled
out.
If $g$ is infinite, then $K'$ is nonempty.  As $K$ is the ends of $T$,
the set $K'$ is a subset of the ends of $T$.  Mark all vertices
on all geodesic rays from $r$ to the ends in $K'$.

\begin{figure}[htb]
\labellist
\endlabellist
\centering
\includegraphics[scale=0.125]{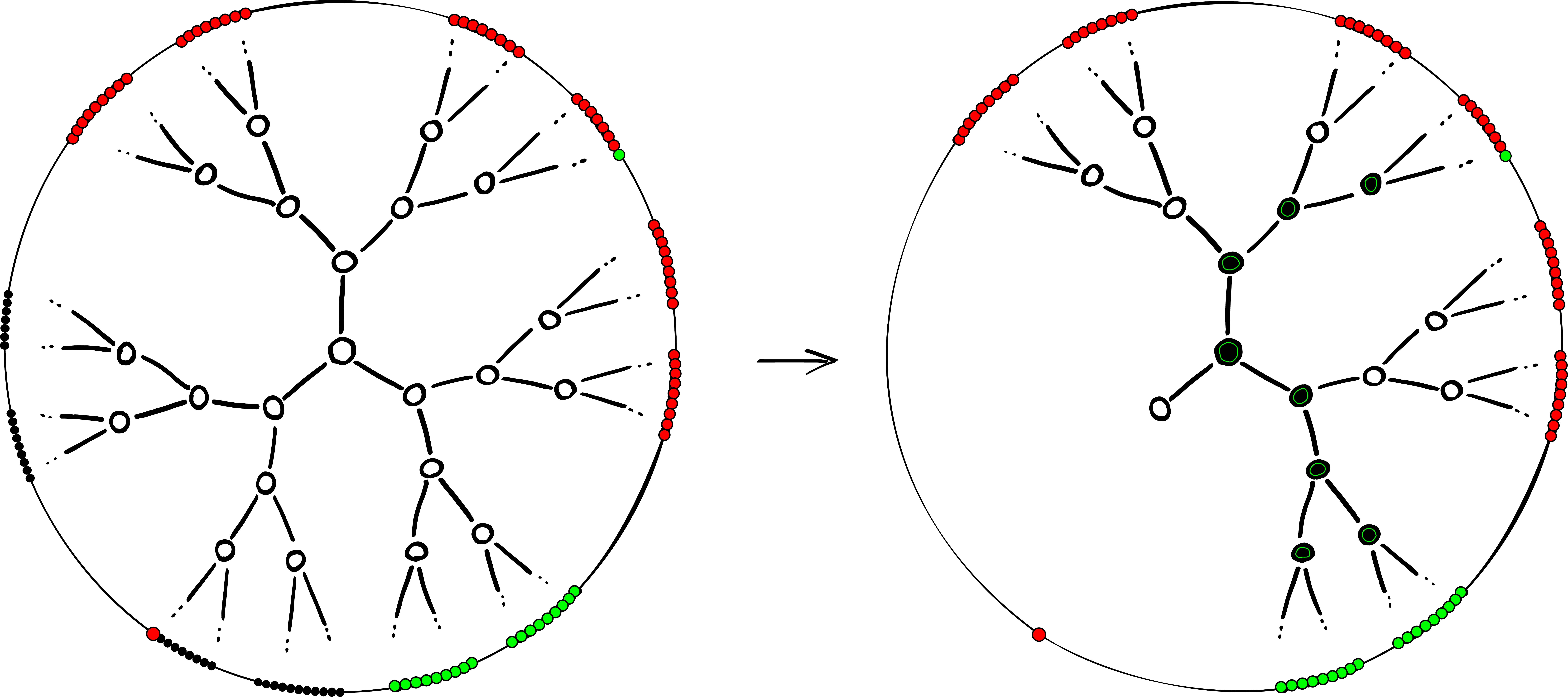}
\caption{Construction of a core tree for a surface $S(g, K,K')$ where: $g$ is infinite, $K=\CC \cup \{p\}$ (in red and green) is a Cantor set union an isolated puncture, and $K'\subset K$ (in green) is a Cantor set union an isolated puncture, both included in $\CC$}
\label{figu:tree}
\end{figure}

Now remove from $T$ all unmarked vertices of valence $2$: if there is an
infinite ray of unmarked valence $2$ vertices (necessarily approaching
an isolated point in $K$ but not $K'$), replace it with a single unmarked valence $1$ vertex.  For any unmarked valence $2$ vertex not on an infinite ray, we can find the maximal segment of unmarked valence $2$ vertices,
 remove it, and attach the vertices on its ends together.
As before, rename $T$ to be this new tree.  Note that the root
$r$ is not removed because by construction it is marked or $3$-valent.

We claim that $T$ is now our desired core tree, and $S(T)\cong S$.
To see this, we just need to check the hypotheses of 
\cite{Richards}, Theorem~1.  Recall $(g,K,K')$ is the triple for $S$,
and let $(h,L,L')$ be the triple for $S(T)$.  
First note that by the construction of $S(T)$, the space $L'$ is
the same as the ends of $T$ accumulated by marked vertices.
Also, $L$ is the same as the ends of $T$ plus the discrete
space of unmarked $1$-valent vertices.
Now, $g=h$ by construction:
we marked exactly $g$ vertices in $T$, or infinitely many if $g$ was infinite.  Next, $(L,L')=(K,K')$ as pairs of spaces.  This was
immediate before we removed unmarked $2$-valent vertices from $T$,
but note that removing any finite chain of such vertices doesn't change
the ends of $T$, and removing any infinite ray of unmarked $2$-valent
vertices removes a single isolated end.  In the construction of $S(T)$
this vertex is replaced by a cusp, so $S(T)$ recovers the
isolated ends in $L$. See Figure \ref{figu:tree} for an example.
\end{proof}

Now we use the core graph to produce a useful embedding of a surface
into $\R^3$ and also a set of proper arcs which cut the surface into
disks.  Recall that a simple proper arc $\alpha$ in a surface $S$
is an embedding
of the open interval $(0,1)$ such that for any compact subset
$C \subseteq S$, the preimage of $C$ under $\alpha$ is compact.
That is, both ends of $\alpha$ must exit out some end of the surface $S$.

\begin{lemma}\label{lem:proper_arcs}
Given a surface $S$ with at least one end and which is not
a disk or annulus, there is a locally
finite collection
$\mathcal W$ of mutually disjoint proper arcs such that the
complement of $\mathcal W$ in $S$ has exactly two components,
and both of them are simply connected.
\end{lemma}
\begin{proof}
Using Lemma~\ref{lem:core_tree}, we can construct a rooted core graph
$T$ so that \mbox{$S(T)\cong S$}.  However, by considering the proof
of the lemma, we can see that more is true: we can take $T$ to
be embedded in $\R^3$ in the vertical plane $\{x=0\}$.  Let
$\Delta$ denote this plane.  Further,
when we are building $S(T)$ and replacing each vertex by the appropriate
subsurface, we can ensure that the surface is cut exactly vertically
in half by $\Delta$.  Note that when we see each subsurface
as constructed out of pants, $\Delta$ intersects each pair
of pants in its seams.  Henceforth, we will equate $S$ and $S(T)$
and imagine that $S$ is itself given by the embedded surface $S(T)$.

The set $\Delta \cap S$ is a set of proper simple arcs and simple closed curves. We denote by $\W_1$ the subset of $\Delta \cap S$  which contains all the arcs of $\Delta \cap S$ and no curve.
Cut $S$ along $\W_1$: the complement $S - \W_1$ has either one or two
connected components. More precisely, if $S$ has positive genus, then at
least one of the elements of $\Delta \cap S$ is a simple closed curve, and
$S-\W_1$ is connected. If $S$ has no genus, then $\Delta \cap S = \W_1$ and
$S-\W_1$ has two connected components, which are both simply connected.
If $\Delta \cap S$ contains no closed curve, then we set $\W=\W_1$ and we
are done.  Otherwise, we continue, and we define sets $\W_2$ and $\W_3$,
as follows.

Note that the construction and rooted structure of
$T$ means that every vertex in $T$ lies on at least one ray from
the root to an end of $T$.  That is, at every vertex there is
at least one edge which points away from the root.  Choose
such an edge at each vertex.  The choice of all these edges
gives a ``combing'' of $T$: if we start at any vertex and follow
these edges, we will follow a geodesic ray towards an end of $T$,
and if any two such paths from two different vertices ever coincide,
they will forever after agree.  Observe that this construction
can produce rays which exit to an unmarked $1$-valent vertex of $T$ (a cusp
of $S(T)$), but for every cusp, only finitely many paths can exit
out that cusp (exactly the paths from the vertices between
the root and the unmarked $1$-valent vertex).

\begin{figure}[htb]
\labellist
\pinlabel $\alpha_2$ at 870 665
\pinlabel $\alpha_3$ at 510 410
\endlabellist
\centering
\vspace{0.2cm}
\includegraphics[scale=0.2]{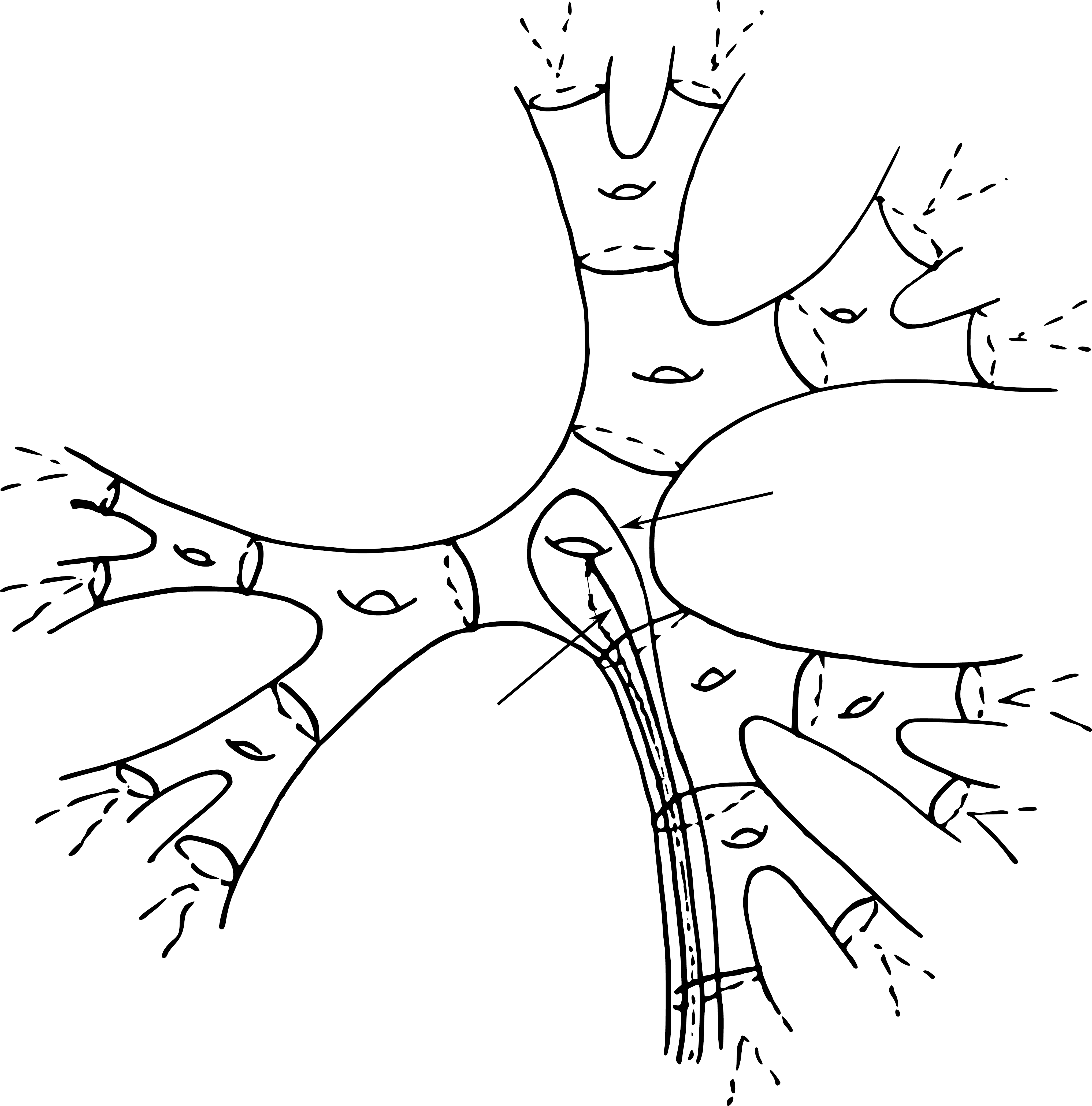}
\caption{Choosing arcs $\alpha_2$ and $\alpha_3$ at every genus in $S$.}
\label{figu:W2}
\end{figure}

\begin{figure}[htb]
\labellist
\endlabellist
\centering
\includegraphics[scale=0.1]{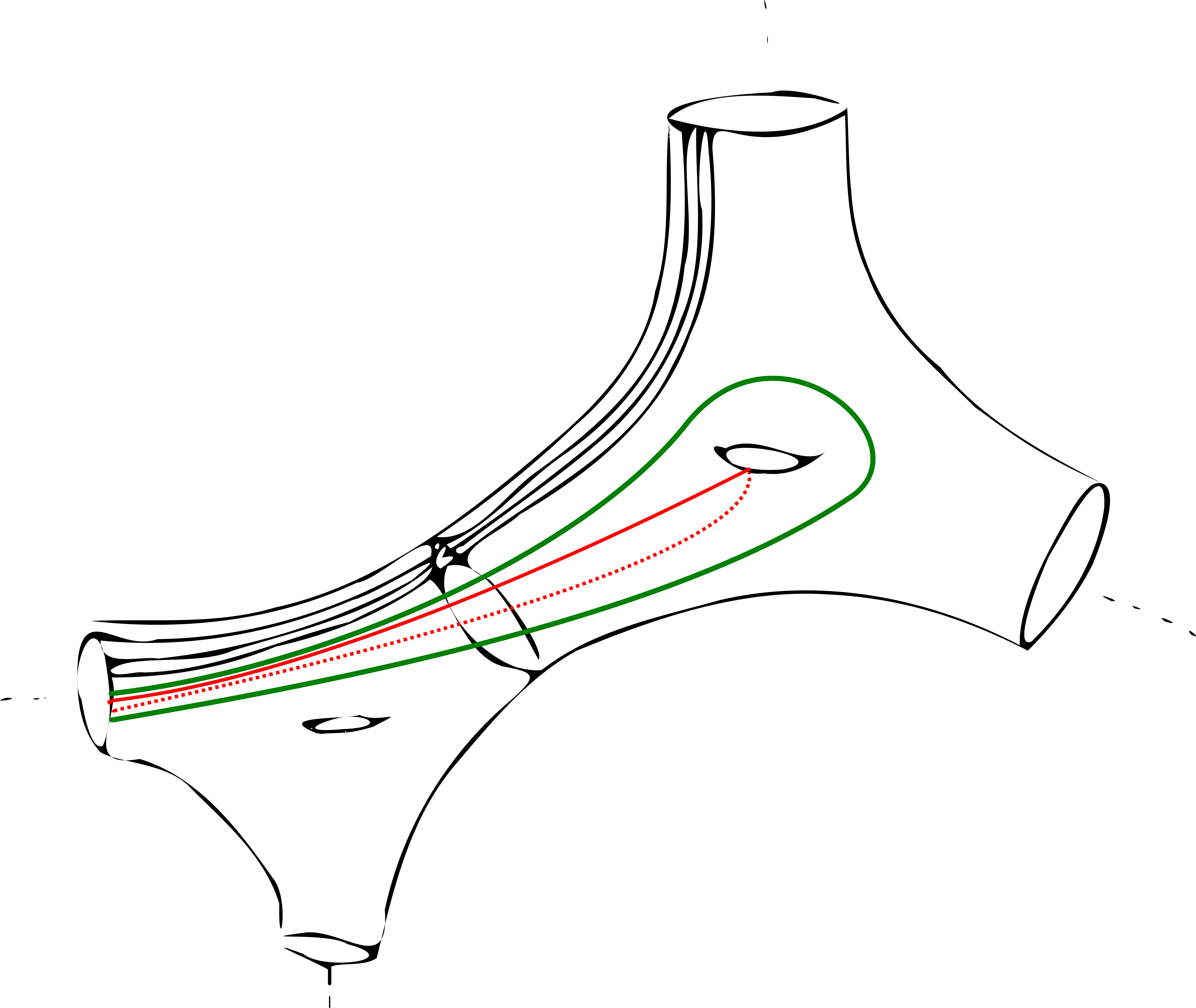}
\caption{Choosing $\alpha_2$ (in green) and $\alpha_3$ (in red)}
\label{figu:W3}
\end{figure}

Now, $\Delta \cap S$ has a genus associated with every marked
vertex in $T$.  
On the associated genus in $S$, choose two arcs $\alpha_2$ and
$\alpha_3$ as in Figure~\ref{figu:W2}, where these arcs
follow the combing we have chosen.  We may choose these arcs
so that $\alpha_3$ intersects $\Delta$ exactly once and $\alpha_2$
does not intersect $\Delta$, as shown.
The collection of all of these arcs can
be chosen to be simple.  This is illustrated by Figure~\ref{figu:W3}
and is because the arcs follow the combing: at every vertex, all
the incoming arcs will turn the same, consistent, way.
Denote by $\W_2$ the set of $\alpha_2$'s and by $\W_3$ the set of
$\alpha_3$'s. 
As discussed above, for any cusp, only finitely many of these
arcs can exit out that cusp, so these sets are locally
finite.

\begin{figure}[htb]
\labellist
\endlabellist
\centering
\includegraphics[scale=0.2]{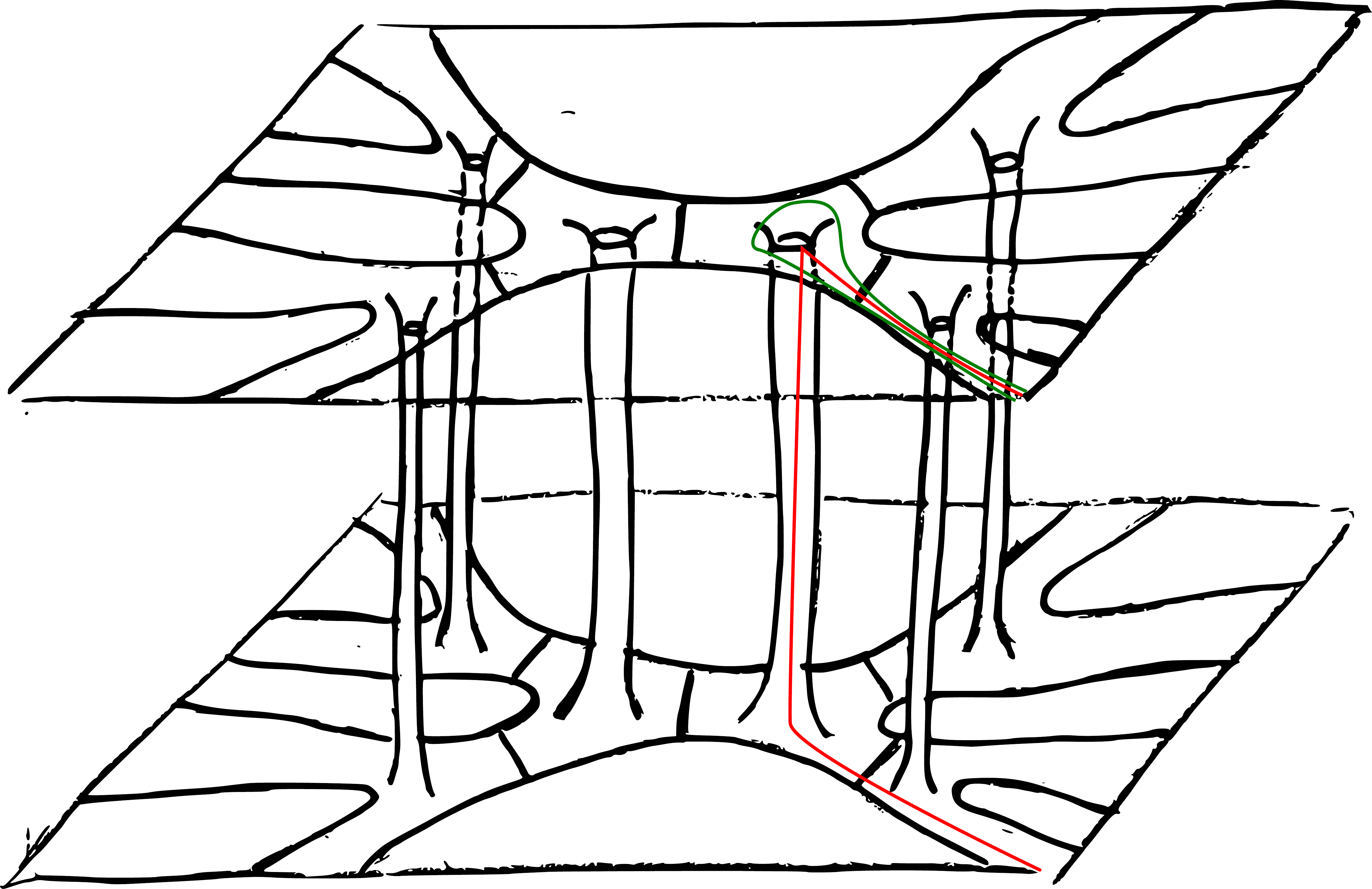}

\caption{The surface $S$ after cutting along $\W_1$, and 
example arcs $\alpha_2$ (green) and $\alpha_3$ (red).}
\label{figu:cut1}
\end{figure}

\begin{figure}[htb]
\labellist
\endlabellist
\centering
\includegraphics[scale=0.15]{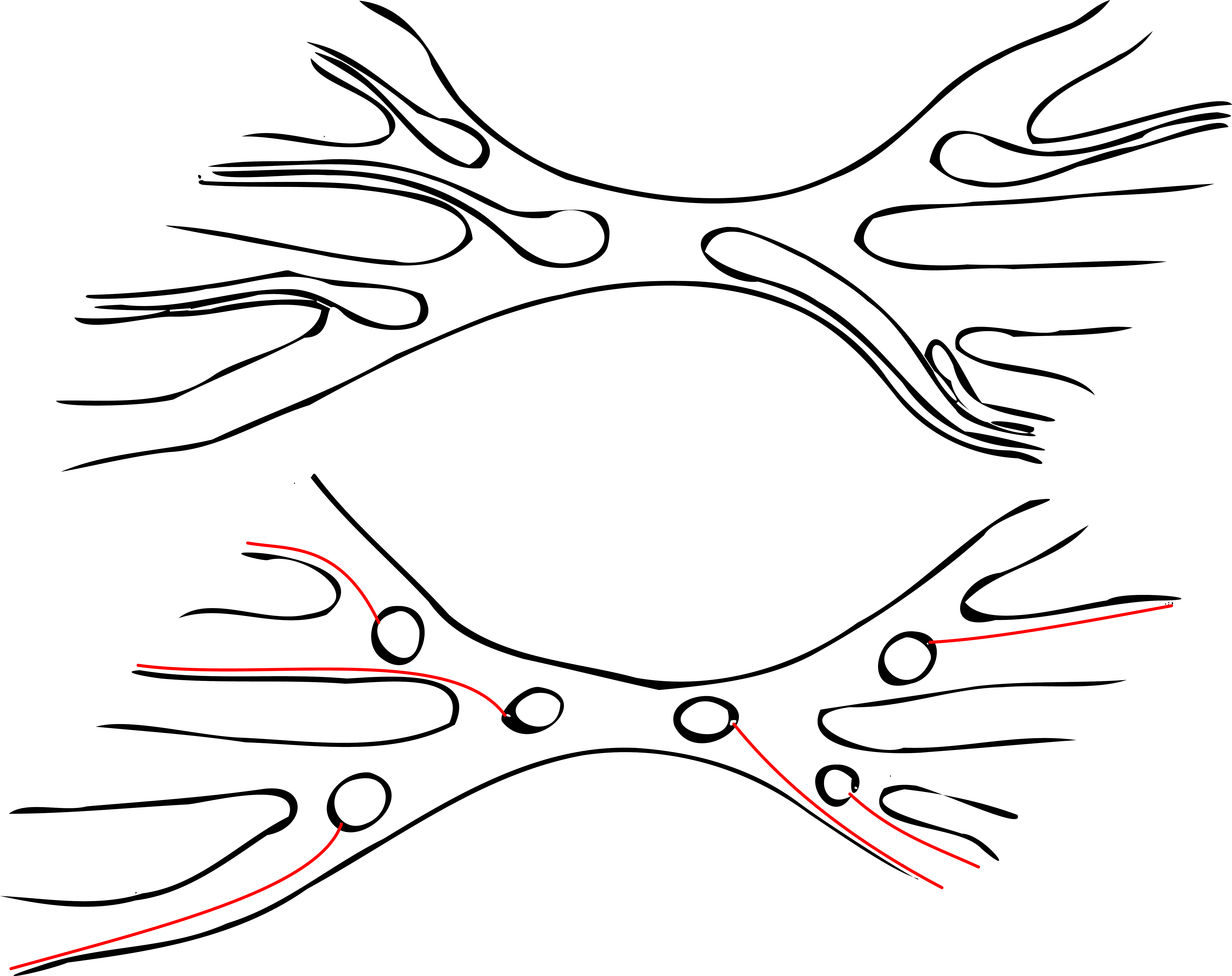}

\caption{The two components of the surface $S$ after cutting along
$\W_1 \cup \W_2$, and the arcs of $\W_3$ (in red)}
\label{figu:cut2}
\end{figure}

Denote by $\W$ the union $\W_1 \cup \W_2 \cup \W_3$. By construction, $\W$
is a locally finite set of proper simple arcs. Moreover, we claim that
the complement of $\W$ has exactly two components, and both of them are
simply connected.  This is illustrated in Figures~\ref{figu:cut1}
and~\ref{figu:cut2}. We get Figure~\ref{figu:cut1} after cutting along
$\W_1$: the surface has two components which are connected by tubes
corresponding to the genus. After cutting along $\W_2$ (the green arcs),
we get Figure \ref{figu:cut2}: the two components are disconnected.
The one on the top is homeomorphic to a disk, and the other one is
homeomorphic to a disk with punctures. Now when we cut along $\W_3$
(the red arcs), the second component becomes homeomorphic to a disk.
\end{proof}

\section{Hyperbolic infinite type surfaces}
\label{section:hyperbolic}

In this section, we make the topological surfaces we constructed
in the previous section rigid by producing a hyperbolic metric.

\begin{theorem}\label{thm:hyperbolic}
Let $S$ be a surface with at least one end which is not a disk or
annulus.  Then $S$ admits a complete hyperbolic metric of the first
kind, and there is a fundamental polygon for $S$ which is geodesic
and has no vertex in the interior of the hyperbolic plane.
Further, every isolated puncture is a vertex on this fundamental polygon,
and for any particular isolated puncture, the polygon can be chosen so that
the Fuchsian group for the polygon has a parabolic element fixing
the puncture.
\end{theorem}
\begin{proof}
The proof is essentially the observation
that we can perform the topological construction of the previous
section in such a way that we produce a hyperbolic surface and the
desired fundamental polygon.  We will follow the proof of
Lemma~\ref{lem:proper_arcs} and describe
how to make it hyperbolic.

Given a surface $S$, we obtain the rooted core tree $T$ from
Lemma~\ref{lem:core_tree}.  As in the proof of
Lemma~\ref{lem:proper_arcs}, we embed $T$ into $\R^3$ such that
it lies in the vertical plane $\{x=0\}$ and all its ends go to
infinity, and then we build the surface $S(T)$ embedded in $\R^3$.
Recall that we can take $S(T)$ to be built out of pairs of pants:
the unmarked vertices are exactly
pairs of pants, but the $1$-, $2$-, and $3$- valent vertices
are replaced by surfaces which can be decomposed into $1$, $2$, and $3$
pairs of pants, respectively.  We can take the intersections of
these pairs of pants with the plane $\{x=0\}$ to be their seams.

It is a basic fact from hyperbolic geometric that there is a unique
hyperbolic pair of pants with specified cuff lengths (the cuff length
is allowed to be zero, meaning that the boundary is a cusp).  Therefore,
specifying that all the cuff lengths of all the pants in $S(T)$ are $1$,
except the cusps which are length $0$, uniquely chooses a
hyperbolic structure on all the pants.  Additionally, we have defined
what the seams of the pants are by their intersection with the vertical
plane.  When we glue the hyperbolic pants to create $S(T)$, we require
that the seams match up.  This completely specifies a hyperbolic
structure on $S(T)$.  Recall the similarities to the Nielsen-Thurston
length and twist parameters to obtain coordinates on the moduli space
of hyperbolic structures of closed surfaces.

It remains to show that this metric is complete and of the first kind.
First note that all the pants composing $S(T)$ are complete metric spaces.
Now consider a closed and bounded subset $C$ of $S(T)$.  Then
$C$ is included in the union of only finitely many of these pairs of pants.
Therefore $C$ is compact.  By the Hopf-Rinow theorem, because $S(T)$
is a connected Riemannian manifold, the metric on it is complete.

Next, observe that the injectivity radius of $S(T)$ is uniformly bounded
above because every point lies on a model pair of pants.  Denote by
$\Omega$ the complement of the closure of the limit set of its
fundamental group in $\partial\H^2$.  If $\Omega$ were
non-empty, then the injectivity radius would be unbounded.  Hence
the metric is of the first kind.

Using the same argument as Lemma~\ref{lem:proper_arcs}, we obtain
a set of proper arcs $\mathcal W$ whose complement has two components
$A$ and $B$,
both of which are simply connected.  Lemma~\ref{lem:isotopy_of_arcs}
allows us to choose these arcs to be geodesic.  These two
disks $A$ and $B$ and the geodesic arcs separating them become our desired
fundamental polygon after choosing any geodesic in $\mathcal W$
and gluing $A$ and $B$ together along it to obtain a single
simply connected region.  Note that by construction, each cusp is cut
in half by the $\mathcal W$ because $\mathcal W$ contains the
two seams which approach the cusp, so each vertex becomes a vertex of
the fundamental polygon.  Furthermore, we can choose any geodesic in
$\mathcal W$ along which to glue $A$ to $B$, so if we are given a
particular cusp $p$, we can choose one of its seams to glue along.
Now $p$ is a vertex of the fundamental polygon, but also one of the
elements of the associated Fuchsian group is a parabolic element
which moves one of the seams at $p$ to the other one and fixes $p$.
\end{proof}

\part{Definition of the two actions and associated tools}

Recall that this part is devoted to generalizing the results
from~\cite{boundary} to the more general class of surfaces
we are considering.  Because many of the arguments
apply verbatim, we will heavily abbreviate our proofs
by relying on the text from~\cite{boundary}.  This makes
the current paper not self-contained, but it avoids
silly duplication.  We hope that indicating the differences
from~\cite{boundary} may actually be more useful, as it highlights
what features are necessary for what arguments.

\section{Conical cover}

\subsection{Equators} 

Let $S$ be a hyperbolic surface with an isolated marked puncture $p$.
By Theorem~\ref{thm:hyperbolic}, there is a geodesic fundamental polygon
for $S$ with no vertex in the interior of $\H^2$ and such that
there is a parabolic element of the Fuchsian group which fixes $p$.
The boundary geodesics in this fundamental polygon (which are two
copies of the arcs $\W$ constructed in the proof of
Theorem~\ref{thm:hyperbolic})
will be very useful.  Let this set of boundary geodesics be the \emph{equator} of $S$.  The name comes from the
special case where $S$ is a sphere minus a Cantor set and an isolated
point; see~\cite{Juliette,boundary}. 
We will soon see how this equator can help us to encode geodesic loops and rays on the surface in a unique way.

\subsection{Conical cover}

\begin{figure}[htb]
\labellist
\endlabellist
\centering
\includegraphics[scale=0.6]{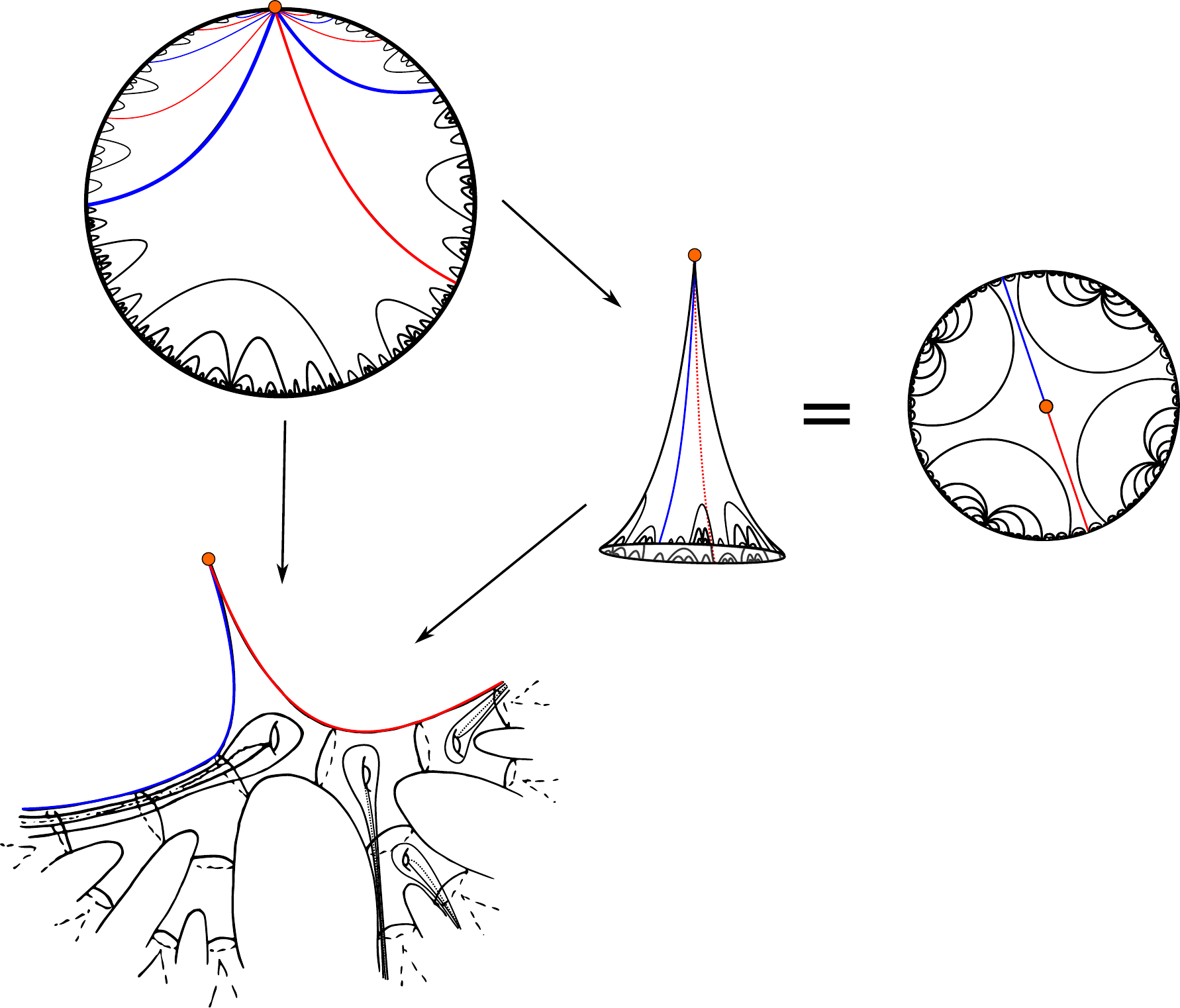}

\caption{A surface with an equator, its universal cover and two versions of the associated conical cover}
\label{figu:conical_cover}
\end{figure}

Our goal is to study rays and loops which start at a marked point $p$
in a surface $S$.  We have constructed a nice fundamental polygon
for $S$, but the universal cover has many copies of this polygon, 
and it's tedious to keep this in mind.  Therefore, we quotient to produce
a simpler \emph{conical cover}.  Let $S$ be a hyperbolic surface with
marked point $p$.  Theorem~\ref{thm:hyperbolic} gives a special
fundamental polygon which has a parabolic element at $p$.
Technically, this parabolic element acts at every lift of $\tilde{p}$:
pick some lift $\tilde{p}$ and quotient $\H^2$ by just this
parabolic element.  The result is the conical cover at $p$.
See Figure~\ref{figu:conical_cover}.  It is homeomorphic to a disk
with a puncture corresponding to $p$.  We denote this cover
by $\pi:\tilde{S}\to S$ and its (circle) boundary by $\partial\tilde{S}$.
Note the critical feature of the conical cover: any ray or loop
from the marked puncture has a unique preferred lift.

\section{Loop graphs}\label{section:loop_graphs}

\subsection{Different types of rays on the surface}
We define short rays, loops and long rays as we did in \cite{boundary}.

\subsubsection{Short rays and loops}
Denote by $E(S)$ the set of ends of $S$.
An embedding $\gamma$ of the open segment $]0,1[$ in $S$ is said to be:
\begin{itemize}
\item A \emph{loop} if it can be continuously extended in $S \cup E(S)$ by $\gamma(0)=\gamma(1)=\{p\}$.
\item A \emph{short ray} if it is proper with one end at $p$.
That is, if it can be continuously extended in $S \cup E(S)$ by $\gamma(0)=\{p\}$ and $\gamma(1) \in E(S)-\{p\}$.
\end{itemize}

We are interested in loops and short rays only up to isotopy. Note that the isotopies fix $E(S)$ pointwise. In particular, all rays in the same isotopy class have the same endpoints. Note also that the set of short rays might be empty (this happens when $p$ is the only end of the surface). Hereafter, we conflate a short ray or loop with its isotopy equivalence class. We say that two short rays or loops are disjoint if there are disjoint rays or loops in their equivalence class (i.e. if they can be made disjoint by separate isotopies). Given a representative of a short ray $r$, there is a distinguished lift $\tilde{r}$ from the base point $\tilde{p}$ in the conical cover.  The ray $\tilde{r}$ limits to some point $q \in \partial\tilde{S}$ (by definition, a lift of the endpoint of $r$ in $E(S)$).  
By Lemma~\ref{lem:isotopy_of_arcs}, the equivalence class of $r$ is specified by $q$, and there is a unique geodesic in the equivalence class of $r$ with limit point $q$ which we can choose as the class representative.  Similarly, given a loop, there is a well-defined limit point in $\partial\tilde{S}$ (which is a lift of $p$, but not the distinguished lift $\tilde{p}$), and we can choose a class representative which is a geodesic in the conical cover.  Note that we conflate this geodesic lift and its image in $S$. 

By construction our equator is locally finite and also finitely many
equator arcs exit to the marked puncture.  Therefore, there are only
finitely many complementary regions locally around $p$.
It is often useful to record a loop or ray combinatorially.  We can do
this by recording which of the finitely many
complementary regions the loop or ray leaves $p$ in, and then
which equator segments it crosses (and in which direction),
and finally which end it
exits the surface out of.  See the further discussion below and
in the definition of $k$-beginning.

\subsubsection{Long rays}

Let $\tilde\gamma$ be a geodesic ray from $\tilde p$ to $\partial\tilde{S}$.  If $\gamma = \pi(\tilde{\gamma})$ is simple and $\gamma$ is not a short ray or a loop, then $\gamma$ is called \emph{long ray}.  Defining a notion of ``isotopically
disjoint'' for topological long rays is difficult, because they don't
have an endpoint in the surface which the isotopy must fix.  We
sidestep the issue by forcing our long rays to be geodesic, and
they are disjoint exactly when these geodesic representatives are
disjoint.  The distinction between short and long rays is their
properness: short rays are proper, while long rays are not.
As with short rays and loops, it is often useful to specify a
long ray with combinatorial information.  Similarly, we can record
a long ray by recording which complementary region around $p$
it leaves from, followed by the sequence of oriented
equator crossings.
Note that because a long ray is not proper, but it is geodesic and limits
to a point on the boundary of the conical cover, it
must cross the equator infinitely many times.

\subsection{Cover-convergence and $k$-beginning}
\label{section:cover_convergence}

The conical cover is useful to us because its boundary is
a circle on which each geodesic from $p$ has an associated limit point.

\begin{definition} We say that a sequence of rays or oriented loops $(x_n)$ \emph{cover-converges} to a geodesic (ray or loop) $l$ on the surface if the sequence of endpoints of the $\tilde{x_n}$ on the boundary of the conical cover converges to a point $q$ such that the image $\pi( (\tilde{p} q) )$ of the geodesic $(\tilde{p} q)$ by the quotient map of the covering is $l$.
\end{definition}

\begin{lemma}\label{lemma:simple_is_compact}
Let $E \subseteq \partial \tilde{S}$ be the set of endpoints of short rays, loops, and long rays (i.e. the set of all endpoints of geodesics from $\tilde{p}$ whose projection to $S$ is simple).  Then $E$ is compact.  The endpoints of loops are isolated (are not accumulation points of $E$).
\end{lemma}
\begin{proof}
The proof of Lemma $2.4.2.$ of \cite{boundary}, where we only need to replace $\infty$ by $p$, gives us that the set of endpoints of simple geodesics from $p$ is compact. 

\end{proof}

\begin{lemma}\label{lemma:cover_converge_limit}
Let $(x_i)$ and $(y_i)$ be sequences of short rays, loops, or long rays such that the pair $x_i$, $y_i$ is disjoint for all $i$.  Suppose that $(x_i)$ cover-converges to $x$ and $(y_i)$ cover-converges to $y$.  Then $x$ and $y$ are disjoint.
\end{lemma}
\begin{proof}
This is Lemma $2.4.3.$ of \cite{boundary}.
\end{proof}

\begin{definition}\label{def:k_begin}
We say that two oriented rays or loops \emph{$k$-begin} like each other
if they leave the marked puncture in the same region and cross the
same initial $k$ elements of the equator in the same direction.
\end{definition}

In Section~\ref{section:action_lox}, we will explicity write
the sequence of equator crossings for a ray; for simplicity, we
will often let the technicality about leaving $p$ in the same
complementary region be understood.  One imagines it could be recorded in the symbol
used for the first crossing.

\begin{lemma}\label{lemma:cover_converge_iff_k_begin}
Let $(x_i)$ be a sequence of rays or loops and let $x$ be a long ray.  Then $(x_i)$ cover-converges to $x$ if and only if for all $k$ there is an $I$ so that for all $i \ge I$, we have $x_i$ $k$-begins like $x$.
\end{lemma}
\begin{proof}
This is Lemma $2.4.5.$ of \cite{boundary}. The same proof works here because we assumed that the equator is a union of \emph{arcs}: equivalently, all the vertices of the chosen fundamental domain are on the boundary of the universal cover.
\end{proof}

\subsection{Loop and ray graphs} 
Here we define the several graphs we will consider in this paper.
\begin{definition} Let $S$ be a surface, and let $X$ be a collection of isotopy classes of simple curves or simple arcs in $S$. We define the following graph $X(S)$ associated to $X$:
\begin{itemize}
\item The vertices of $X(S)$ are the elements of $X$.
\item The edges of $X(S)$ are the pairs of homotopically disjoint elements of $X$.
\end{itemize}
\end{definition}

For example, the usual curve graph is $X(S)$ when $X$ is the set of isotopy classes of simple essential closed curves of $S$. 

In this paper, we will consider:
\begin{itemize}
\item The \emph{loop graph} $X(S)$, where $X$ is the set of isotopy classes of simple loops based on $p$. We denote it by $L(S;p)$.
\item The \emph{short-ray graph} $X(S)$, where $X$ is the set of isotopy classes of (simple) short rays based on $p$. We denote it by $R_s(S;p)$.
\item The \emph{short-ray-and-loop graph} $X(S)$, where $X$ is the set of isotopy classes of (simple) short rays and loops based on $p$. We denote it by $R_sL(S;p)$.
\item The \emph{completed ray graph} $X(S)$, where $X$ is the set of isotopy classes of (simple) short rays, loops and long rays based on $p$. We denote it by $\RGC(S;p)$.
\end{itemize}

We will prove that for any surface $S$ with at least one isolated puncture $p$, the loop graph $L(S;p)$ is Gromov-hyperbolic and quasi-isometric to the short-ray-and-loop graph $R_sL(S;p)$ and to the main connected component of the completed ray graph $\RGC(S;p)$. However, it is in general not true that the short-ray graph $R_s(S;p)$ is Gromov-hyperbolic (although it is true in the case of the plane minus a Cantor set studied in \cite{boundary} and \cite{Juliette}).  Because we always have a particular surface $S$ and
marked point $p$ in mind, we will often suppress the $p$ in the notation, 
e.g. writing $\RGC$ for the completed ray graph.

\subsection{Hyperbolicity of loop graphs} 
The following result was proved in a more general setting by Aramayona-Fossas-Parlier in \cite{Aramayona-F-P}. We give here another proof adapted to our setting. 

\begin{theorem}
Let $S$ be a connected surface with at least one isolated puncture $p$. The loop graph $L(S;p)$ is Gromov-hyperbolic.
\end{theorem}

\begin{proof}
This theorem was proved in \cite{Juliette} for the specific case where $S$ is the plane minus a Cantor set ($p$ was then equal to $\infty$, and the loop graph was denoted by $X_\infty$). The proof followed the arguments given by Hensel, Przytycki and Webb in their proof of the uniform Gromov-hyperbolicity for curve complexes of finite type surface in \cite{Hensel-Przytycki-Webb}.
The same proof works in the general case (with the exact same arguments than Section~$3.1$ of \cite{Juliette}).
\end{proof}

\subsection{Actions}

\begin{lemma}\label{lemma:mod_acts}
The mapping class group $\Mod(S;p)$ acts by homeomorphisms on the boundary $\partial \tilde{S}$ of the conical cover.  It acts by isometries on $\RGC(S;p)$.
\end{lemma}
Note that even if every homeomorphism of $S$ lifts to the conical cover, it is a priori not obviously possible extend a homeomorphism of the cover to its boundary. Indeed, geodesic rays of the cover are not necessarily mapped to geodesic rays by homeomorphisms.
\begin{proof}
This is Lemma $2.6.1$ of \cite{boundary}. Although the proof is almost verbatim the same, because this action is central for us we recall it here in our context. First, we observe that $\Mod(S;p)$ preserves the set of all simple and non simple loops.  As subset of the boundary of the conical cover, the set of (oriented) simple and non simple loops is dense: this is because lifts of $p$ are dense. Next, as we require that any mapping class $\phi \in \Mod(S;p)$ fixes $p$, there is a well-defined lift $\tilde{\phi}:\tilde{S} \to \tilde{S}$ fixing $\tilde{p}$.  As any mapping class must preserve the cyclic order of geodesic rays from $\tilde{\infty}$, the mapping class must preserve the cyclic order of the dense set of loops endpoints and thus extend to a homeomorphism on the boundary $\partial\tilde{S}$.  To avoid issues of isotopy, our definition of a long ray is in terms of endpoints on $\partial\tilde{S}$.  Hence only now, after establishing that $\tilde{\phi}$ induces a homeomorphism on $\partial\tilde{S}$, can we observe that this implies that $\phi$ acts on the set of rays.  Furthermore, this action preserves disjointness between long rays and short rays and loops (because this can be defined in terms of cyclic orders on the boundary $\partial\tilde{S}$).  Thus $\phi$ induces an isometry on $\RGC(S;p)$.
\end{proof}

Note that it is not true in general that $\Mod(S;p)$ acts transitively on short rays. Indeed, if the set of punctures is not a Cantor set, $S$ can have different types of punctures (for example, some could be isolated, and some others could be not isolated). It is not possible to map a ray which ends on an isolated puncture to a ray which ends on a non isolated one. Hence Lemma $2.6.2.$ of \cite{boundary} is not true in the general case.

\subsection{Filling rays}

We define several notions of fillingness for long rays. Note that the definition of $k$-filling is different from \cite{boundary} (where $l_0$ was a ray).

\begin{definition}
A long ray $l$ is said to be:
\begin{itemize}
\item \emph{loop-filling} if it intersects every loop.
\item \emph{ray-filling} if it intersects every short ray.
\item \emph{k-filling} if 
       \begin{enumerate}
       \item There exists a loop $l_0$ and long rays $l_1,\ldots, l_k=l$ such that $l_i$ is disjoint from  $l_{i+1}$ for all $i\ge 0$.
       \item $k$ is minimal for this property.
      \end{enumerate}
\item \emph{high-filling} if it is ray-filling and not $k$-filling for any $k \in \N$.
\end{itemize}
\end{definition}

That is, a long ray is $k$-filling if it is at distance exactly $k$ from the set of loops in the graph $\RGC(S;p)$.  A long ray is high-filling if it is not in the connected component of $\RGC(S;p)$ containing the short rays.

\subsubsection{Some properties of filling rays}

\begin{lemma} \label{lemma:disjoint_from_loop_filling}
Let $L$ be a loop-filling ray.
If $l$ and $l'$ are two rays (short or long) disjoint from $L$, then $l$ and $l'$ are disjoint.
\end{lemma}
\begin{proof}
This is Lemma $2.7.4$ of \cite{boundary}.
\end{proof}

It is not true in general that any ray-filling ray is also loop-filling: indeed, if a surface has genus, we can easily imagine  a long ray which spirals along the genus and isolates $p$ from the other punctures. Such a ray will be ray-filling, but not necessary loop-filling. Hence the first part of Lemma $2.7.2$ of \cite{boundary} is false in general.

\begin{corollary} \label{corollary:disjoint_from_loop_filling}
Let $l$ be a loop-filling ray which is not ray-filling, then every long ray disjoint from $l$ is not ray-filling.
\end{corollary}
\begin{proof}
This is Corollary $2.7.5.$ of \cite{boundary}.
\end{proof}

\begin{lemma}\label{lemma:no_k_filling}
There exists no $k$-filling ray for $k> 2$.
\end{lemma}
\begin{proof}
This is Lemma $2.7.6.$ of \cite{boundary}, where we replace $l_0$ by a loop.
\end{proof}

\begin{lemma}\label{lemma:cliques}
Any connected component of $\RGC(S;p)$ containing a high-filling ray is a clique of high-filling rays.
\end{lemma}
\begin{proof}
This is Lemma $2.7.8.$ of \cite{boundary}.
\end{proof}

\begin{lemma}\label{lemma:cliques_compact}
As a subset of the boundary of the conical cover, every clique of high-filling rays is compact.
\end{lemma}
\begin{proof}
This is Lemma $2.7.9.$ of \cite{boundary}.
\end{proof}

\subsection{Connected components of $\RGC$}

\begin{theorem}\label{theorem:completed_components}
The connected component of the completed loop graph $\RGC(S;p)$ containing all loops is quasi-isometric to the loop graph $\L(S;p)$.  All other connected components (which are the high-filling rays) are \emph{cliques}, i.e. complete subgraphs.
\end{theorem}

We start with a technical lemma, whose statement is the same than Lemma~$2.8.2.$ of \cite{boundary}. However, the proof in \cite{boundary} was specific to the plane minus a Cantor set. We give here a more general proof.

\begin{lemma}\label{lemma:short_loops_enough}
Let $a$ and $b$ be two loops at distance $3$ in the completed ray graph. Then they are at distance $3$ in the short-ray-and-loop graph.
\end{lemma}
\begin{proof}
Let $(a,\lambda,\mu,b)$ be a geodesic path in the completed ray graph.  We will show how to replace $\lambda$ and $\mu$ with loops, which will prove the lemma.  This proof proceeds by chasing down various cases.  Figure~\ref{fig:short_loops_enough} provides an illustration.

\begin{figure}[htb]
\labellist
\pinlabel $a$ at 10 100
\pinlabel $b$ at 15 160
\pinlabel $\lambda$ at 45 125
\pinlabel $\mu$ at 120 162

\pinlabel $s_b$ at 170 150
\pinlabel $\lambda$ at 185 162
\pinlabel $\hat{\lambda}$ at 205 162

\pinlabel $C$ at 310 60
\pinlabel $\lambda$ at 300 162
\pinlabel $\hat{\lambda}$ at 317 162
\pinlabel $\mu$ at 358 103
\pinlabel $\hat{\mu}$ at 355 88
\endlabellist
\includegraphics[scale=0.98]{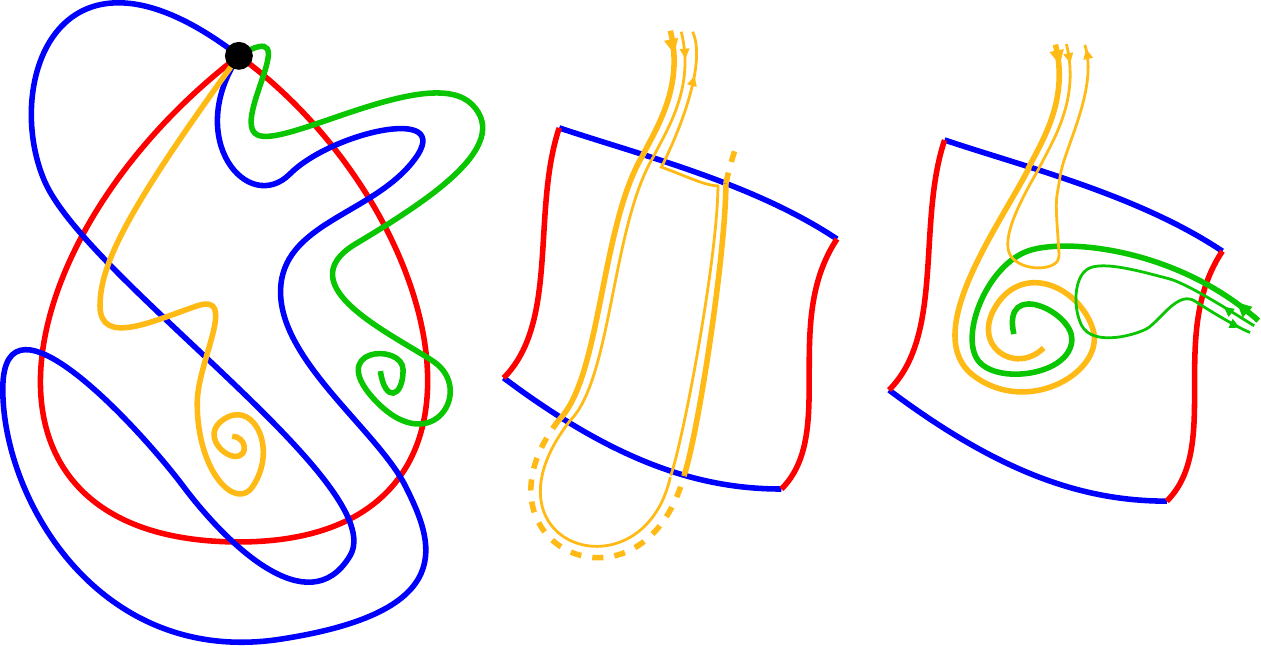}
\caption{An illustration of Figure~\ref{lemma:short_loops_enough}.  If $\lambda$
intersects some component of $b-a$ more than once, we can replace it with a loop.
Otherwise, it might enter some component of the complement of $a \cup b$ and not exit.
In this case, in the most general situation, we replace $\lambda$ and $\mu$ with
loops.}
\label{fig:short_loops_enough}
\end{figure}
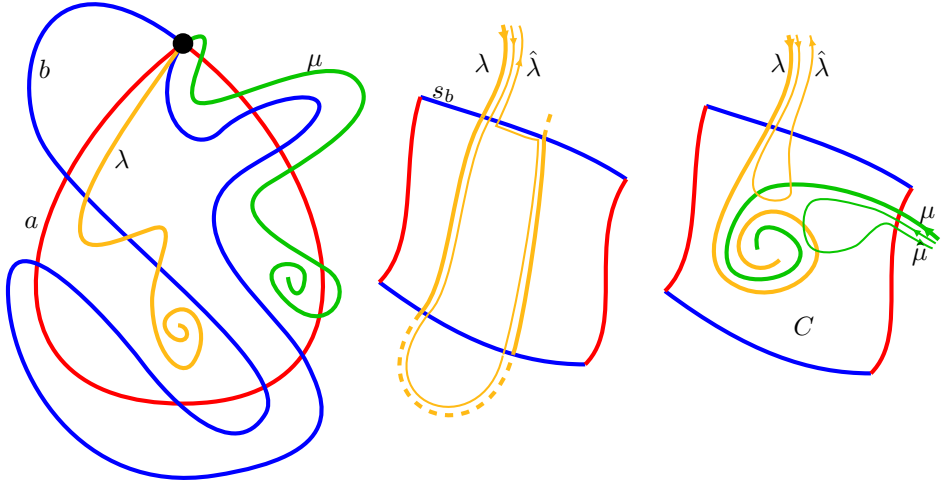

Since $a$ and $b$ are loops, they intersect finitely many times.  Therefore $b-a$ is
composed of finitely many segments.  The long ray $\lambda$ may intersect these segments,
since $\lambda$ is not disjoint from $b$.  There are two cases depending on whether
$\lambda$ intersects some segment more than once or intersects all segments once.

Suppose that $\lambda$ intersects some segment of $b-a$ more than once.  Call this
segment $s_b$.  Let $x_1$ and $x_2$ be the first two points of intersection between
$\lambda$ and $s_b$ as we travel from $\infty$ along $\lambda$.  Because $\mu$ is disjoint from $b$, it is in particular disjoint from the subsegment $(x_2x_1)_{s_b}$ of $s_b$ between $x_2$ and $x_1$. The union $(\infty x_2)_\lambda \cup (x_2 x_1)_{s_b} \cup (x_1 \infty)_\lambda$ is an essential loop $\hat \lambda$ disjoint from $a$ and $\mu$.

By applying the same procedure to $\mu$, if $\mu$ intersects any segment of $a-b$ more than
once, we can replace it with an essential loop $\hat{\mu}$ disjoint from $\lambda$ and $b$.  Note these replacements can be made simultaneously.  See Figure~\ref{fig:short_loops_enough},
center.

We must now address the cases in which $\lambda$ and/or $\mu$ intersect all segments
of $b-a$ (respectively, $a-b$) at most once.  Consider $\lambda$.  There must be some segment of 
$b-a$ which it crosses last, say at a point $x_0$.  Thus, the entire ray from $x_0$
along $\lambda$ is disjoint from $a,\mu$, and $b$, and in particular, it is trapped
in a component $C$ of $S - (b\cup a)$.  In the most general case, $\mu$ is also trapped in
this same component.  Because $\lambda$ and $\mu$ are geodesic long rays,
there must be either genus or at least two punctures in $C$ accessible for both $\lambda$
and $\mu$.  We can therefore resolve both of then into loops $\hat{\lambda}$,
$\hat{\mu}$ with the appropriate
disjointness properties.  Note this case applies even when $\lambda$ or $\mu$ has
already been replaced with a loop.  See Figure~\ref{fig:short_loops_enough}, right.
\end{proof}

\begin{proof}[Proof of Theorem~\ref{theorem:completed_components}]
This is Theorem $2.8.1.$ of \cite{boundary}.
\end{proof}

\begin{remark}
After Theorem~\ref{theorem:completed_components}, we will conflate the definitions of the loop graph and the short-ray-and-loop graph, using whichever is most convenient.  We'll refer to all of these graphs, and also to the quasi-isometric connected component of $\RGC$, by $\RG$.  We'll also call this component the \emph{main component}, as distinguished from the \emph{cliques}.  We show in Section~\ref{section:bijection} that the
cliques are in bijection with the Gromov boundary of $\RG$.
\end{remark}

\section{Gromov-boundaries of loop graphs}
\label{section:bijection}

\subsection{Infinite unicorn paths}\label{section:unicorn_paths}
This section is the adaptation of Section $3$ of \cite{boundary}. 

\subsubsection{Finite unicorn paths}
\label{section:finite_unicorn_paths}
In \cite{boundary}, we give two definitions of unicorn paths in loop graphs (the first is adapted from \cite{Hensel-Przytycki-Webb}), and we prove that these two definitions are equivalent. We use here the exact same definitions, without recalling them (Definitions $3.2.1$ and $3.2.2$ of \cite{boundary}).

Lemmas $3.2.3$, $3.2.4$, $3.2.5$, $3.2.6$, $3.2.7$ of \cite{boundary} remain true with the exact same proofs.  We only have to be careful with the proof of Lemma $3.2.7$, which uses the equator, but it is still true with our new definition of the equator.

\subsubsection{Infinite unicorn paths}
We define infinite unicorn paths as we did in \cite{boundary}, Definition $3.3.1$ (which was adapted from \cite{Pho-On}).
Lemmas $3.3.3$, $3.3.4$ of \cite{boundary} remain true, with the same proofs.

\subsubsection{Infinite unicorn paths and cover-convergence}

Lemma $3.4.1$ of \cite{boundary} remains true, with the same proof.

\subsubsection{Infinite unicorn paths and filling rays}

Lemmas $3.5.1$ and $3.5.2$ of \cite{boundary} remain true, with the same proofs. The proof of Lemma $3.5.3.$ of \cite{boundary} used the old defintion of ray-filling and the Cantor set. Hence the proof in the general case is slightly different. We give here a new proof for the general case.

\begin{lemma}\label{lemma:unicorn_path_1_ray_filling} [Lemma $3.5.3.$ of \cite{boundary}]
Let $l$ be a $2$-filling ray. For every oriented loop $a$, $P(a,l)$ is bounded. Moreover, there exists $a$ such that $P(a,l)$ is included in the $2$-neighborhood of $a$.
\end{lemma}
\begin{proof}
By definition of $2$-filling ray, there exists a ray $l'$ disjoint from $l$ with a loop $a$ disjoint from $l'$. Consider any $x_k=(p p_k)_l \cup (p p_k)_a$ in $P(a,l)$. We will find a loop $y_k$ disjoint from $a$ and $x_k$: hence $x_k$ is at distance at most two from $a$.

Denote by $l_k$ the sub-segment $(p p_k)_l$ of $l$ between $p$ and $p_k$. Consider $l_k \cup a$: this is a compact set of $S$, whose complement has finitely many connected components. Each component of the complement of $l_k \cup a$ is arcwise connected. As $l'$ is disjoint from both $l$ and $a$, it is included in a connected component of the complement of $l_k \cup a$. Denote this connected component by $C$. Note that $C$:
\begin{itemize}
\item is not simply connected, hence contains some points of $\P$ or some genus (because $l'$ is a geodesic included in $C$);
\item contains $p$ in its boundary (because it contains the ray $l'$).
\end{itemize}
It follows that we can draw a loop $y_k$ included in $C$: this loop is disjoint from both $a$ and $x_k$. Hence $P(a,l)$ is in the $2$-neighborhood of $a$. 

According to the adaptation of Lemma $3.3.3$ of \cite{boundary}, for every oriented loop $b$, $P(b,l)$ is included in the $d(a,b)$-neighborhood of $P(a,l)$, hence $P(b,l)$ is bounded.
\end{proof}

Lemmas $3.5.4.$, $3.5.5.$ and $3.5.6.$ of \cite{boundary} remain true, with the same proofs.

\subsection{Gromov-boundary and unicorn paths}

Everything in Section~4 of~\cite{boundary} holds. 
In particular, Lemmas $4.3.1$, $4.4.1$ and $4.5.1$
remain true in the general case, with the same proofs.

\subsection{Bijection}

\begin{theorem}\label{theorem:boundary_bijection}
Given a point $x$ on the Gromov boundary $\partial \RG$, there is a nonempty set of long rays $R$ such that the point $x$ is the equivalence class of $P(a,l)$ for all $l \in R$ and any loop $a$.  The set $R$ is a clique (and an entire connected component) of high-filling rays in $\RGC$.  Conversely, given a connected component $R$ of high-filling rays in $\RGC$ (which is necessarily a clique), there is a single boundary point $x \in \partial \RG$ such that $x$ is the equivalence class of $P(a,l)$ for all $l \in R$ and any loop $a$, and $R$ is exactly the set of rays with this property.

Hence the set of high-filling cliques is in bijection with the boundary of the loop graph $\partial \RG$.
\end{theorem}

\begin{proof} This is Theorem $5.1.1.$ of \cite{boundary}. Its proof is still valid here.
\end{proof}

Following Theorem~\ref{theorem:boundary_bijection}, we can define a bijection between cliques in $\RGC$ and the boundary of $\RG$.  Let $\HC$ be the set of high-filling rays, and let $\E$ be the set of cliques of high-filling rays in the graph $\RGC$.  If $l$ is a high-filling ray, we let $[l]$ denote the clique in $\E$ containing $l$.  By Theorem~\ref{theorem:boundary_bijection}, 
we can define the map $F: \E \to \partial\RG$ by $F([l]) = [P(a,l)]$, and $F$ is a bijection.  That is, we take a clique in $\E$ to the equivalence class of $P(a,l)$ on the boundary, for any $l$ in the clique and any loop $a$.

There is a lifted map $\tilde{F} = F\circ q:\HC \to \partial\RG$, 
where $q$ is the map $q:\HC \to \E$ taking a long ray to its clique.  See Figure~20 in~\cite{boundary}.

\subsection{Topologies}
We define topologies on both the set of high-filling rays and the Gromov
boundary of the loop graph as we did in \cite{boundary} (Sections $5.2.1.$ and $5.2.2.$).  As in \cite{boundary}, the topology on the high-filling rays induces a topology on the set of cliques in $\RGC$, and it the space $\RGC$ endowed with this topology which we will show is homeomorphic to the Gromov boundary. Lemma $5.2.2.$ of \cite{boundary} is true in the general case, with the same proof.

\subsection{High-filling rays and neighborhoods}

Lemma $5.3.1$ of \cite{boundary} is true in the general case, with the same proof.
Because we refer to it later, we restate it here for convenience.
\begin{lemma}[Lemma~5.3.1 of \cite{boundary}]
\label{lemma:control_of_neighborhoods_of_geodesics_general_m}
Let $l$ be a high-filling ray. Let $m\in \N$. For every $k \in \N$, there
exists $N \in \N$ such that every loop at distance at most $m$ from a loop
which $N$-begins like~$l$ must $k$-begin like one of the rays in the clique of $l$. 
\end{lemma}

We re-write the proof of the following proposition (the idea is similar than the proof of \cite{boundary} -- the modification comes from the fact that we don't necessary have a loop which intersects the equator only once in the general case).

\begin{proposition}[Proposition $5.3.2.$ of \cite{boundary}]\label{prop:neighborhoods_on_boundary}
Let $P \in \partial \RG$. For every $k \in \N$, there exists $N>0$ such that for every $Q \in U(P,N)$, for every $\lambda' \in \tilde{F}^{-1}(Q)$, there exists $\lambda \in \tilde{F}^{-1}(P)$ which $k$-begins like $\lambda'$.
\end{proposition}

\begin{proof} Consider $P \in \partial \RG$ and choose $k \in \N$.
Let $\lambda \in \tilde{F}^{-1}(P)$. Note that $\lambda$ is high-filling, and $[P(a,\lambda)]=P$, and the clique of high-filling rays disjoint from $\lambda$ is $\tilde{F}^{-1}(P)$.  These are all consequences of Theorem~\ref{theorem:boundary_bijection}.

According to the generalization of Lemma $5.3.1.$ of \cite{boundary} with $m=2\delta + 20$, there exists $N' \in \N$ such that every loop at distance less than $2 \delta +20$ from a loop which $N'$-begins like $\lambda$ must $k$-begin like one of the rays of $\tilde{F}^{-1}(P)$. 

Pick a loop $a$ and denote by $s$ its number of intersection points with the equator. Denote by $(x_n)$ the infinite unicorn path $P(a,\lambda)$. Because $a$ has $s$ intersections with the equator, we have that for any $k'>0$, for any $n> k's$, $x_n$ at least $k'$-begins like $\lambda$ by the generalization of Lemma $3.2.7.$ of \cite{boundary}. Hence $x_{sN'+1}$ $N'$-begins like $\lambda$. Let us prove that $N:=sN'+1+2\delta$ satisfies the proposition.

Choose $Q \in U(P,N'+2\delta)$ and $\lambda' \in \tilde{F}^{-1}(Q)$. Again by Theorem~\ref{theorem:boundary_bijection}, we have $\lambda'$ high filling and $[P(a,\lambda')]=Q$.  Denote by $(y_n)$ the unicorn path $P(a,\lambda')$. By definition of $U(P,N'+2\delta)$, we have $\liminf_{i,j \rightarrow \infty} (x_i \cdot y_j)_a \geq N'$. Choose $i>sN'+1$ and $j \in \N$ such that $(x_i \cdot y_j)_a \geq N'$: any two geodesics $[a,x_i]$ and $[a,y_j]$ have their $(sN'+1)$ first terms $(2 \delta)$-Hausdorff close.

Because geodesics and unicorn paths are uniformly close (Generalization of Lemma~$3.2.5.$ of \cite{boundary}), it follows that some point of $(y_n)$, say $y_p$, is in the $2\delta+20$ neighborhood of $x_{sN+1}$. According to the generalization of Lemma $5.3.1.$ again, $y_p$ $k$-begins like one of the rays in $\tilde{F}^{-1}(P)$. Thus $\lambda'$ also $k$-begins like one of the rays in $\tilde{F}^{-1}(P)$. 
\end{proof}

\subsection{Homeomorphism}

This section is similar to section $5.4$ of \cite{boundary}. Lemmas $5.4.1.$ and $5.4.2.$ remain the same, and we get:

\begin{theorem}\label{theorem:boundary_homeo}
The map $F : \E \to \partial \RG$ is a homeomorphism.
\end{theorem}

\begin{proof} This is Theorem $5.4.3.$ of \cite{boundary}.
\end{proof}

\part{Simultaneous dynamics of the actions and applications}

\section{Action of loxodromic elements on the circle}\label{section:action_lox}

In the previous sections, we proved that the loop graph of a surface and its boundary embed in a circle (the boundary of the conical cover), on which the mapping class group $\Mod(S;p)$ acts by homeomorphisms (Lemma \ref{lemma:mod_acts}).
We now want to study how this particular representation of $\Mod(S;p)$
in $\Homeo(S^1)$ is related to the action of $\Mod(S;p)$ on the loop graph
$L(S;p)$. More precisely, we prove that elements with a loxodromic
action on the loop graph necessarily have a simple dynamics on the
circle: they have finitely many periodic points (and thus rational
rotation number). The situation is as in Figure~\ref{figu:action_boundary}:
each loxodromic element fixes two points on the boundary of the
conical cover, and thus two cliques of high filling rays: we prove
that these cliques are finite, and have the same number of elements,
which are the periodic points for the action on the circle.

\begin{figure}[htb]
\labellist
\pinlabel $h$ at 160 210
\pinlabel $\textrm{Loop graph }\L(S;p)$ at 200 -40
\pinlabel $\textrm{Boundary of the conical cover}$ at 750 -40
\endlabellist
\centering
\includegraphics[scale=0.3]{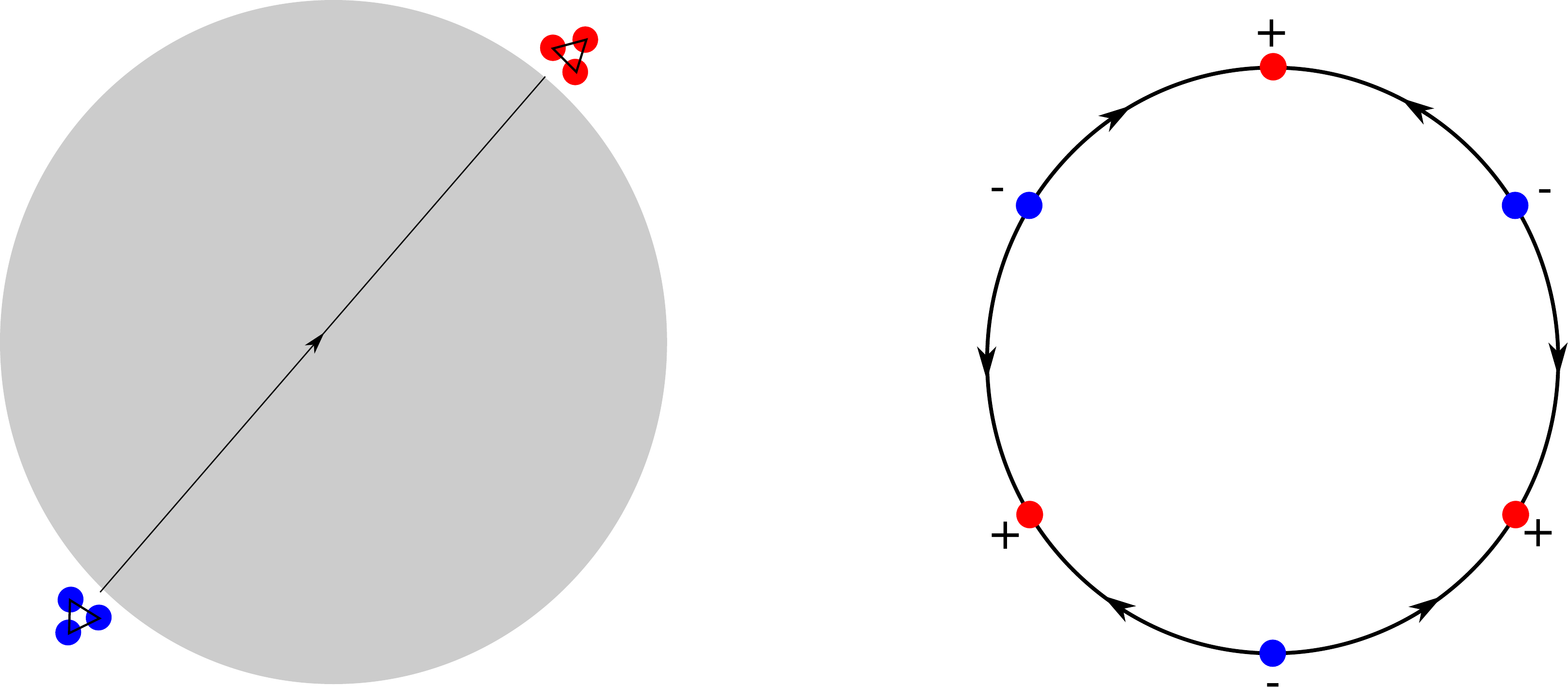}
\vspace{1cm}
\caption{On the left, the action of a loxodromic element $h$ on the loop graph: the two fixed points on the boundary are cliques. On the right, the action of $h^m$ (for $m$ minimal so that $h^m$ fixes its finite cliques pointwise) on the boundary of the conical cover: the elements of the two attractive and repulsive cliques associated to $h$ give the attractive and repulsive points for this circular action.}
\label{figu:action_boundary}
\end{figure}

\subsection{Finite cliques}
Recall that $\Mod(S;p)$ acts by isometries on the loop graph $\L(S;p)$. Let $h \in \Mod(S;p)$ be an element with a loxodromic action on $\L(S;p)$. This element fixes exactly two points on the boundary of $L(S;p)$. Equivalently, it preserves two cliques of high-filling rays. We denote by $\CC^-(h)$ and $\CC^+(h)$ respectively the repulsive and the attractive cliques of~$h$. 
We say that a homeomorphism $g$ of the circle has a \emph{Morse-Smale dynamics} if it has finitely many fixed points, which are all repulsive or attractive (as in the right side of Figure \ref{figu:action_boundary}).

\begin{theorem}\label{theorem:finite cliques}
Let $h \in \Mod(S;p)$ be a loxodromic element. The cliques $\CC^-(h)$ and $\CC^+(h)$ of high-filling rays associated to $h$ are finite and have the same number of elements.
Moreover, if we identify $S^1$ as the boundary of the conical cover,
then there exists $k\in \N$ such that the action of $h^k$ on $S^1$
has a Morse-Smale dynamics with fixed points exactly the
high-filling rays of the attractive and repulsive cliques. 
\end{theorem}

Because the proof involves some technical lemmas, we delay it until after be discuss some consequences and context.
First, we immediately conclude that the circle action induced
by a loxodromic homeomorphism has a rational rotation number.

\begin{corollary}
Let $h \in \Mod(S;p)$ with a loxodromic action on the loop graph $\L(S;p)$. As a
homeomorphism of the boundary $S^1$ of the conical cover, $h$ has a
rational rotation number.
\end{corollary}

Next, because the size of the cliques $\CC^\pm(h)$ is finite, it can be
used to differentiate between homeomorphisms with the following definition.
\begin{definition}
Let $h \in \Mod(S;p)$ be a loxodromic element. The \emph{weight} of $h$
is the number of elements of each clique associated to $h$ (i.e. the
cardinality of $\CC^\pm(h)$).  We denote the weight of $h$ by $w(h)$.
\end{definition}

Observe that $w(h) = w(h^{-1})$ and that $w(h)$ is a conjugacy invariant.
In Section~\ref{section:subsurfaces}, we will show that $w(h)$ is unbounded for
an infinite type surface and uniformly bounded for finite type surfaces. In Section~\ref{section:quasimorphisms},
we will show that if $G$ is a subgroup of $\Mod(S;p)$ which contains two
loxodromic elements $g$ and $h$ with $w(g) \neq w(h)$, then the space
$\tilde Q(G)$ of non trivial quasimorphisms on $G$ is infinite dimensional.\\

We now proceed to the proof of Theorem~\ref{theorem:finite cliques}.
The first step is to show that every high-filling ray is accumulated on
both sides by loops in the boundary of the conical cover.  Note that
this is actually nontrivial, and the issue is the simplicity (simple-ness)
of the approximating loops.  As an example of the complexity, note that loops are actually isolated (in the set of simple rays and loops)
in the boundary of the conical cover, because to
get close to a loop, an approximating ray or loop must spiral around
a puncture and then return, which can only be done by a geodesic which intersects itself.

Let $\lambda$ be a ray; it crosses a (possibly empty) sequence of equator
arcs.  Denote by $\lambda_n$ the sequence of arcs, and denote by $\lambda_n^+$
and $\lambda_n^-$ the right and left endpoints, respectively, of these arcs
on the boundary of the conical cover.
See Figure~\ref{figu:equator_crossings}.
Note that $\lambda_n$ refers to an equator
crossing in the surface $S$, but $\lambda_n^\pm$ refer to endpoints of
lifts of the equator arcs.  For example a long ray might spiral around a closed
simple loop (not based at $p$); in this case the sequence $\lambda_n$ would be periodic,
but the sequences $\lambda_n^\pm$ would converge to one of the endpoints of a lift
of the simple loop.  Also note that many equator arcs can share a single endpoint.

\begin{figure}[htb]
\centering
 \includegraphics[scale=0.7]{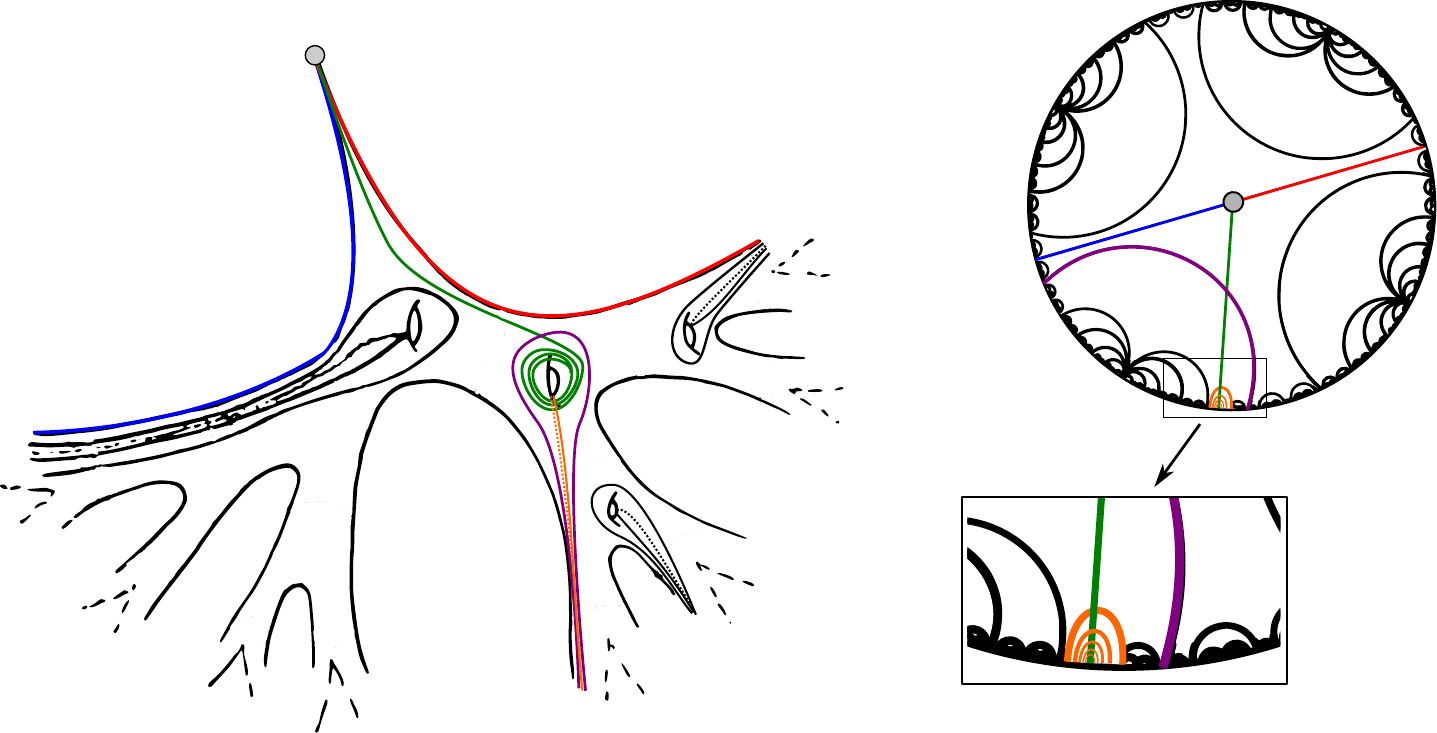}

\caption{In green, a long ray $\lambda$, its equator crossings $\lambda_n$, and the endpoints
$\lambda_n^+$ and $\lambda_n^-$.}
\label{figu:equator_crossings}
\end{figure}

\begin{lemma}\label{lemma:high_filling_equator_sequence}
In the above notation, if $\lambda$ is high-filling, then
the sequences $\lambda_n^+$ and $\lambda_n^-$ are infinite and have
infinitely many distinct elements.
In addition, the sequence $\lambda_n$ is not preperiodic.
\end{lemma}

\begin{proof}
First note that the sequence $\lambda_n$ is infinite: if it were finite, then
the supposedly high-filling ray would exit the surface $S$ out a puncture or
end (because it can not not leave the last fundamental domain it enters).
In particular, if we take any loop $\ell$, then the number of intersections
between $\lambda$ and $\ell$ is finite, which by the unicorn path
construction means that there is a finite path in $\RGC$ between $\ell$ and $\lambda$,
which contradicts that $\lambda$ is high-filling.

Thus the sequences $\lambda_n^\pm$ are infinite.  Note that if we have $\lambda_n^+=\lambda_m^+$
for some $m>n$, then $\lambda_n^+ = \lambda_{n+1}^+ = \cdots = \lambda_m^+$, and
similarly for $\lambda_n^-$  This is
because a geodesic cannot backtrack in the conical cover.  The only way to have repeated
elements of the sequence $\lambda_n^\pm$ is to cross a sequence of equator arcs in a row, all
with the same endpoint.  Now suppose that $\lambda$ is high-filling and 
$\lambda_n^+$ or $\lambda_m^-$ (wlog $\lambda_n^+$) has finitely many distinct elements.
By the above argument, this means that $\lambda_n^+$ is eventually constant, which means
that $\lambda$ simply spirals around the endpoint cusp or end which $\lambda_n^+$ represents,
which contradicts that $\lambda$ is high-filling.

Finally, suppose that $\lambda_n$ is preperiodic, say
\[
(\lambda_n)_n = (\lambda_1, \ldots, \lambda_{m-1}, \overline{\lambda_m, \ldots, \lambda_{m+k}})
\]
Then $\lambda$ eventually spirals
around the closed loop defined by the finite periodic sequence $\lambda_m, \ldots, \lambda_{m+k}$.
Then the loop at $p$ which does the crossings $\lambda_1, \ldots, \lambda_{m+k}$,
followed by the reversed crossings $\lambda_{m-1}, \ldots, \lambda_1$ is simple and disjoint
from $\lambda$, which contradicts that $\lambda$ is high-filling.
\end{proof}

\begin{lemma}\label{lemma:loop_approximation}
Let $\lambda \in S$ be a high-filling ray.
There are two sequences of loops $x_i$ and $y_i$
such that as points on the boundary of the conical cover,
the sequences $x_i$ and $y_i$ limit to $\ell$ from the
left and right, respectively.
\end{lemma}
\begin{proof}
The sketch of the proof is that in order to produce a loop
very close to $\lambda$ on the left (wlog), we will follow the
beginning of $\lambda$ and then diverge to close up the loop.
We must ensure two things: first, the loop must actually
be to the left in the boundary of the conical cover, and second,
the loop must be simple.

Let us be given $k$; we will construct a loop to the left
of $\lambda$ which $k$-begins like $\lambda$, and this will prove
the lemma.  Consider the complement of the union of the equator
(all arcs in the equator) and the
sub-ray which is the beginning of $\lambda$ until the intersection
with the equator crossing $\lambda_k$.  
We will denote the intersection
of the ray $\lambda$ and the equator segment $\lambda_k$ by
$u_k$, so $\lambda$ is the sequence of segments
$(p, u_1)$, $(u_1, u_2), \ldots$.
Denote by $\Omega_k$
the component of this complement which is on the left of the subsegment
$(u_{k}, u_{k+1})$.  In particular, if we are trying to
construct a loop $x$, and $x$ follows $\lambda$ until $u_k$,
then it arrives in $\Omega_k$.  

The boundary of the closure of $\Omega_k$ may contain an equator arc
$s$ which is distinct from $\lambda_k$ and $\lambda_{k+1}$ and
to the left of $\lambda$, or it may not.
If it does not, we claim that there is some $k' > k$ 
such that the boundary of the closure of $\Omega_{k'}$ 
does in fact contain such an equator arc.  This is essentially
a consequence of Lemma~\ref{lemma:high_filling_equator_sequence}:
if the closure of the boundary of $\Omega_k$ does not contain
another equator arc, the reason is either that 
\begin{enumerate}
\item $\lambda_k$ and $\lambda_{k+1}$ are the only equator arcs
to the left of $\lambda$.  In this case, the endpoints $\lambda_{k}^-$
and $\lambda_{k+1}^-$ are the same.
\item Another lift of the initial segment of $\lambda$
cuts away all the other equator segments to the left of $\lambda$.
\end{enumerate}
By Lemma~\ref{lemma:high_filling_equator_sequence}, the endpoint
sequence $(\lambda_n^-)_n$ has infinitely many distinct elements,
so there is some $k'>k$ for which (1) does not hold.  If case (2)
holds, then consider this other lift $\tilde{\lambda}$.
The two lifts cannot have the same endpoints, because by
Lemma~\ref{lemma:high_filling_equator_sequence} the
sequence $(\lambda_n)_n$ is not preperiodic.  Thus, if we follow
the lift $\tilde\lambda$ until it diverges from $\lambda$ (at
a $k' > k$), case (2) cannot hold, and we can find a suitable $k'$.

Once we have a $k'$ such that the closure of the boundary of
$\Omega_{k'}$ does contain an equator arc $s$ disjoint
from $u_{k'}$ and $u_{k'+1}$ and to the left of $\lambda$,
we can extend the (simple arc) beginning of $\lambda$ to cross
$s$.  That is, there is a simple arc from $p$ to $s$ with equator
crossings $(u_1, u_2, \ldots, u_k, \ldots, u_{k'}, s)$.

Denote this arc by $x$.  The union of the arc $x$ and the equator
arc $s$ is a tripod.  If we remove $x \cup s$ from $S$, we
add $2$ to its (negative) Euler characteristic, but we do not
disconnect the surface.  Thus, the surface $S - (x \cup s)$
has enough complexity to have an essential simple loop $z$.
Note here we are using the convenience assumption we made
in Section~\ref{sec:surfaces} that $\chi(S) < -2$.
Then connect $x$ to $z$ with a simple arc $y$ in $S - (x \cup s \cup z)$.
Then the sequence $x,y,z,\bar{y}, \bar{x}$, where $\bar{x}$
means the reverse of $x$, gives a simple loop which is
forced to $k$-begin like $\lambda$ and lie to the left.
\end{proof}

\begin{proof}[Proof of Theorem~\ref{theorem:finite cliques}]
By Lemma~\ref{lemma:cliques_compact}, the cliques $\CC^+(h)$
and $\CC^-(h)$ are disjoint compact subsets of $\partial \tilde{S} = S^1$.
Let $A = S^1 \setminus (\CC^+(h) \cup \CC^-(h))$ be the complement
of the union of the cliques.  The only open sets in $S^1$ are intervals,
so $A$ is a union of intervals.  Write $A = \bigcup_{c\in C} I_c \cup \bigcup_{d \in D} I_d$, 
where each $I$ is an interval, and $C, D$ index all the intervals whose
endpoints lie in the same clique or in different cliques, respectively.  That is,
for all $d \in D$, we have one endpoint of $I_d$ in $\CC^+(h)$ and one endpoint
in $\CC^-(h)$. 

We claim that $|D|$ is finite.  To see this, suppose it is infinite, so we can produce
an infinite sequence $(d_i)_i$ such that each $I_{d_i}$ is distinct.  Let the endpoints
of $I_{d_i}$ be denoted $a_i$ and $b_i$.  By taking a subsequence (and flipping
the endpoints, wlog), we may assume that $a_i \in \CC^+(h)$ and $b_i \in \CC^+(h)$
for all $i$.  Because $S^1$ is compact, we may take further subsequences and assume that
$a_i$ and $b_i$ are convergent.  Consider the length of $I_{d_i}$ in the metrized
standard unit circle.  Because the $I_{d_i}$ are distinct, the lengths must go to $0$.
Therefore $\lim_{i\to\infty}a_i = \lim_{i\to\infty} b_i$.  Because $\CC^+(h)$ and $\CC^-(h)$
are compact and thus closed, this limit point must be in both $\CC^+(h)$ and $\CC^-(h)$,
but these sets are disjoint, so we have a contradiction, and we conclude that $|D|$ is finite.

Since $|D|$ is finite, there is some power $h^k$ such that $h^k(I_d) = I_d$
for all $d \in D$.  The union $\bigcup_{d\in D} I_d$ is a finite union of open
intervals, so we can write the complement as
$S^1 \setminus \bigcup_{d\in D} I_d = \bigcup_{e \in E}I_e$, where $|E|$ is finite
(in fact, $|E|=|D|$ because the intervals alternate)
and each $I_e$ is a closed interval whose endpoints are both in $\CC^+(h)$
or both in $\CC^-(h)$.  The interval $I_e$ can contain only points in $\CC^+(h)$
or points in $\CC^-(h)$, but not both (because that would imply the existence of
another interval $I_d$).  To clarify, note that the intervals $I_e$ are not the
same as $I_c$.  Since $h^k(I_d) = I_d$ for all $d\in D$, we also have $h^k(I_e) = I_e$
for all $e \in E$.  We claim that each $I_e$ consists of a single point.  Given
any $I_e$, let its endpoints be $a$ and $b$.  Wlog assume they are in $\CC^-(h)$, so
any point in $\CC^+(h) \cup \CC^-(h)$ which lies in $I_e$ must be in $\CC^-(h)$.
By Lemma~\ref{lemma:loop_approximation}, if $a \ne b$, then there is a loop $\ell$
whose endpoint on the boundary on the conical cover lies between $a$ and $b$.
Because $h$ is loxodromic, $h^{nk}(\ell)$ must converge to some point
in $\CC^+(h)$ as $n \to\infty$.  Also, $h^k(I_e) = I_e$, so $h^{nk}(\ell) \in I_e$
for all $n$.  Therefore, the limit of $h^{nk}(\ell)$ is a point in $\CC^+(h)$
which lies in $I_e$, which contradicts our earlier assertion about $I_e$ containing
only points in $\CC^-(h)$.  The only way out of the contradiction is if $I_e$ is
in fact a single point.

Thus we have proved that there are finitely many intervals $I_d$, each of which has one
endpoint in $\CC^+(h)$ and one endpoint in $\CC^-(h)$, and the complement of
the $I_d$ is a finite set of points.  This implies that the cliques $\CC^\pm(h)$
are finite and have the same number of points.  It remains to show the Morse-Smale
dynamics on the complementary intervals $I_d$.  To see this, we apply
Lemma~\ref{lemma:loop_approximation} again; this implies that for each $I_d$,
there is some loop $\ell \in I_d$.  Acting by $h^k$ must cause $\ell$ to 
limit to the endpoint in $\CC^+(h)$, and acting by $h^{-k}$
must cause $\ell$ to limit to the other endpoint, which lies in $\CC^-(h)$.
That is, exactly the Morse-Smale dynamics as claimed.
\end{proof}

\subsection{Links between cliques and laminations}
In the finite type case, there is a well known description of the boundary
of the curve complex in terms of minimal filling laminations. In this section,
we relate this description with our description of the loop graph in terms of cliques.

Let $S$ be a finite type surface with a preferred puncture $p$. Let $h$ be a
pseudo-Anosov element of $S$ which fixes $p$. Then $h$ preserves two transverse
foliations $\mathcal{F}^+$ and $\mathcal{F}^-$ on $S$, and the mapping class of
$h$ in $\Mod(S;p)$ preserves two cliques of high-filling $\CC^+(h)$ and $\CC^-(h)$
on $S$. These objects are related in the following way:

\begin{lemma}
The leaves of $\mathcal{F}^+$ (resp. $\mathcal{F}^-$) ending in $p$ are the
high-filling rays of $\CC^+(h)$ (resp. $\CC^-(h)$).
\end{lemma}
\begin{proof}
Consider a conical cover of $S$. Lift $\mathcal{F}^+$ and $\mathcal{F}^-$ to
the conical cover. The fixed points of the action of $h$ on the boundary
of the conical cover are exactly the elements of $\CC^+(h)$ and $\CC^-(h)$.
Moreover, $\CC^+(h)$ is attractive and $\CC^-(h)$ is repulsive.
\end{proof}

In this setting, the minimal filling laminations $\Lambda^+$ and $\Lambda^-$
associated to $h$ can be defined by:
$$\Lambda^\pm = clos(\lambda^\pm)-\lambda^\pm$$
for any $\lambda^\pm \in \CC^\pm(h)$, where $clos(\lambda)$ is the closure of $\lambda$ in $S$.

We could describe the boundary of the loop graph of any orientable surface
(e.g. of infinite type) $S$ with a preferred isolated puncture $p$ with
laminations in the same way: this boundary is exactly the set of distinct
laminations of the form $clos(\lambda)-\lambda$ for some high-filling ray
$\lambda$ in $S$.  It would be interesting to explore these objects.  However, there
are two main difficulties.  First, we don't have any description of the
laminations that appear in this way other than the definition above
induced by the high-filling rays.  In particular, we don't have a good
notion of what ``minimal and filling'' means.
Secondly, we use the action of $\Mod(S;p)$ on the boundary of the conical
cover, where each (high-filling) ray has a unique lift: if we want to lift
laminations that do not necessarily end at the puncture $p$, we loose
the ability to choose a preferred lift.

\section{Embeddings of surfaces, loop graphs and boundaries}
\label{section:subsurfaces}
This section studies the embeddings and projections of loop graphs
via subsurface projections (Sections~\ref{section:subsurfaces_1}
and~\ref{section:subsurfaces_2}) and an application of these maps
(Section~\ref{section:subsurfaces_3}).
The main ideas of Sections~\ref{section:subsurfaces_1}
and~\ref{section:subsurfaces_2} are already in \cite{Masur-Minsky,Schleimer,Aramayona-F-P}.  Our main purpose is to produce 
examples of loxodromic actions on infinite type surfaces by embedding
actions from finite type surfaces.  In some sense, these are the least
interesting cases, since when studying infinite type surfaces we
want to explore actions which are intrinsically infinite type and do
\emph{not} come from finite type dynamics.  See~\cite{Juliette}
for an example.  However, finite type actions are easier to construct;
in this section, we use them to show that the set of weights of elements
of $\Mod(S;p)$ is unbounded for infinite type $S$.

\subsection{Loops: embedding and projection}
\label{section:subsurfaces_1}

\begin{definition}
Let $S$ and $S'$ be two surfaces with isolated marked punctures $p$ and $p'$.  We say that $(S;p)$ is \emph{essentially embedded in $(S';p')$} if there exists an embedding $\phi:S\rightarrow S'$ such that:
\begin{itemize}
\item $p$ is mapped to $p'$.
\item If $q$ is another puncture of $S$, either $q$ is mapped to a puncture of $S'$, or for any curve $\gamma$ around $q$, $\phi(\gamma)$ is essential in $S'$.
\end{itemize}
\end{definition}
That is, an essential embedding is allowed to ``open up'' a puncture,
as long as it goes around an essential loop.

Let $S$ be a finite type surface with an isolated marked puncture $p$. Let $S'$ be
a surface (of finite or infinite type) with marked puncture $p'$, and
suppose that there exists an essential embedding
$\phi : (S;p) \rightarrow (S';p')$.
We will define a quasi-isometric embedding $\phi^\rightarrow$ from
$\bar\L(S;p)$ to $\bar\L(S',p')$ and a $3$-Lipschitz projection
$\phi^\leftarrow$ from $\bar\L(S',p')$ to $\bar\L(S;p)$.  Here $\bar\L(S;p)$ is the loop graph together with its Gromov boundary, i.e.
it is all loops and long rays.  Because
essential embeddings are allowed to open up punctures, short
rays cause technical difficulties.  Fortunately,
by Theorem~\ref{theorem:completed_components}, the loop graph is quasi-isometric
to the short ray and loop graph, so we can simply ignore the short rays.

\begin{definition}
If $\gamma \in S$ is a loop or ray in $S$ based at $p$, then
$\phi(\gamma)$ is a loop or ray in $S'$.  This induces the map
$\phi^\rightarrow: \bar\L(S;p) \to \bar\L(S;p)$.
\end{definition}

\begin{definition}
Define the projection
$\phi^\leftarrow:\bar\L(S';p') \to \bar\L(S;p)$ as follows.
We choose a geodesic representative for every boundary
component of $\phi(S)$ (the punctures which have been opened up).  We denote this set of curves by $\Gamma$.
Given $\alpha \in \bar\L(S';p')$, we choose a geodesic representative
of $\alpha$.  If $\alpha$ is included in $\phi(S)$,
then set $\phi^\leftarrow(\alpha) = \alpha$, understood to be
a ray or loop in $S$.  Otherwise, 
$\alpha$ has some first intersection $x_1$ with $\Gamma$,
which is on a curve $\gamma$.  Define
$\phi^\leftarrow(\alpha)$ to be the union of the
segments $(p', x_1)_\alpha \cup (x_1,x_1)_\gamma \cup (x_1, p')_\alpha$.  Since the boundaries of $S$ are not homotopically
trivial, $\phi^\leftarrow(\alpha)$ is indeed an essential
embedded loop in $S$.
\end{definition}

\begin{lemma}\label{lemma:proj_is_proj}
The map $\phi^\leftarrow \circ \phi^\rightarrow$,
is the identity, and $\phi^\rightarrow \circ \phi^\leftarrow$ is
idempotent.
\end{lemma}
\begin{proof}
For any $\alpha \in \bar\L(S;p)$, the curve $\phi^\rightarrow(\alpha)$
is included in $\phi(S)$, so by definition $\phi^\leftarrow(\phi^\rightarrow(\alpha)) = \alpha$.
This also implies that $\phi^\rightarrow \circ \phi^\leftarrow$ is
idempotent, because $\phi^\rightarrow \circ (\phi^\leftarrow \circ 
\phi^\rightarrow) \circ \phi^\leftarrow = \phi^\rightarrow \circ \phi^\leftarrow$.
\end{proof}

\begin{lemma}\label{lemma:lipschitz}
The map $\phi^\leftarrow$
is $3$-Lipschitz.
\end{lemma}
\begin{proof}
Let $\alpha,\beta \in \bar\L(S';p')$ be disjoint.  Intersections
between 
$\phi^\leftarrow(\alpha)$ and
$\phi^\leftarrow(\beta)$ can only occur if $\alpha$ and
$\beta$ intersect $\Gamma$ and also have their first
intersection with $\Gamma$ on the same component, 
and at most two intersections can be created in this way.
In this case, the images
$\phi^\leftarrow(\alpha), \phi^\leftarrow(\beta)$ are loops.
By the construction of unicorn paths, then,
$d(\phi^\leftarrow(\alpha), \phi^\leftarrow(\beta)) \le 3$.
If $\alpha,\beta$ are at distance $n$, we can repeat this
argument on every pair of consecutive loops in a path
between $\alpha$ and $\beta$ to prove the lemma.
\end{proof}

\begin{lemma}\label{lemma:cliques_to_cliques}
If $\CC \subseteq \bar \L(S;p)$ is a maximal clique of high-filling rays,
then $\phi^\rightarrow(\CC)$ is a maximal clique of high-filling rays.
\end{lemma}
\begin{proof}
Let us be given $\ell \in \CC$.  We first show that any ray or
loop which is disjoint from $\phi^\rightarrow(\ell)$ in $S'$
is contained in $\phi(S)$.  Towards a contradiction, suppose
the opposite: that we have a ray or loop $\gamma \in S'$,
such that $\gamma$ is disjoint from $\phi^\rightarrow(\ell)$, and
assume that $\gamma \notin \phi(S)$.  Then
by Lemma~\ref{lemma:lipschitz}, there is a path of length
at most $3$ in $\bar\L(S;p)$ between
$\phi^\leftarrow(\gamma)$ and $\phi^\leftarrow(\phi^\rightarrow(\ell)) = \ell$.
Because $\gamma \notin \phi(S)$, we know that $\gamma$
must intersect $\Gamma$, so $\phi^\leftarrow(\gamma)$ is a loop.
Thus we have a path of length $3$ between a loop and a high-filling
ray, which is a contradiction.

We know that $\phi^\rightarrow(\CC)$ is a clique of rays, and we have
just shown that any ray or loop disjoint from any of them is contained
in $\phi(S)$, but $\CC$ is maximal in $\bar\L(S;p)$, so $\phi^\rightarrow(\CC)$ is maximal.
\end{proof}

\begin{lemma} \label{lemma:qi_embed}
The map
$\phi^\rightarrow:\bar\L(S;p) \to \bar\L(S';p')$ is a quasi-isometric
embedding.
\end{lemma}
\begin{proof} 
Clearly we have:
\[
d(\phi^\rightarrow(\alpha),\phi^\rightarrow(\beta)) \leq d(\alpha,\beta)
\]
because any path in $\bar\L(S;p)$ is mapped by $\phi^\rightarrow$
to a path in $\bar\L(S';p')$.
Conversely, let
\[ 
(\phi^\rightarrow(\alpha),
\gamma_1,...,\gamma_k,
\phi^\rightarrow(\beta))
\]
be a geodesic path in $\L(S';p')$ between two elements of
$\phi^\rightarrow(\bar\L(S;p))$.  Now apply the projection
$\phi^\leftarrow$ to the path.  By Lemmas~\ref{lemma:lipschitz} and~\ref{lemma:proj_is_proj}, the endpoints map back to $\alpha$
and $\beta$, and consecutive entries map under $\phi^\leftarrow$ to loops
and rays
which are at most distance $3$ in $\bar\L(S;p)$.
Thus we conclude
\[
d(\alpha,\beta) \le
3d(\phi^\rightarrow(\alpha), \phi^\rightarrow(\beta)).
\]
The combination of these inequalities shows that
$\phi^\rightarrow$ is a $(3,0)$-quasi-isometric embedding.
\end{proof}

\subsection{Boundaries: embedding}
\label{section:subsurfaces_2}

By Lemma~\ref{lemma:qi_embed}, it follows
from the definition of the Gromov boundary
that the map $\phi^\rightarrow$ induces an injective
continuous map on the boundaries
$\phi^\rightarrow : \partial L(S;p) \to \partial L(S';p')$.
However, in this section, we show the stronger
statement below.

\begin{lemma}\label{lemma:commutes}
The (injective, continuous) map $\phi^\rightarrow$ commutes with
the homeomorphisms
$F_S: \{\textnormal{cliques in $\bar\L(S;p)$}\} \to \partial L(S;p)$
and
$F_{S'}: \{\textnormal{cliques in $\bar\L(S';p')$}\} \to \partial L(S';p')$
given in Theorem~\ref{theorem:boundary_homeo}.
\end{lemma}
\begin{proof}
The map $F_S$ can be defined by, given a clique $\CC$
in $\bar\L(S;p)$, the image $F_S(\CC)$ (in the Gromov boundary
of $\L(S;p)$) is represented by the equivalence class of the 
unicorn path $P(a,\ell)$ for any loop $a$ and any ray $\ell \in \CC$.
If we take $a$ to be a loop in $\phi(S)$,
then evidently the entire unicorn path construction is mapped
via $\phi^\rightarrow$ to a unicorn path $P(\phi^\rightarrow(a),
\phi^\rightarrow(\ell))$, whose equivalence class represents
the image $F_{S'}(\phi^\rightarrow(\CC))$.  That is, the lemma follows
from the fact that the unicorn path construction commutes with 
$\phi^\rightarrow$.
\end{proof}

\subsection{Weights}
\label{section:subsurfaces_3}

\begin{theorem}\label{theo:weights}
Let $S$ be an infinite type surface with an isolated marked puncture $p$. For any $n\in \N$,
there exist elements of $\Mod(S;p)$ with a loxodromic action on the loop graph
$\L(S;p)$ whose weight is equal to $n$.
\end{theorem}

\begin{proof}
(The notations $S$ and $S'$ are reversed here; now $S'$ is the subsurface).
Let $S'$ be an essential subsurface of $S$ with enough complexity so that there
exists a pseudo-Anosov $h \in \mathrm{Homeo}(S')$ whose preserved foliations
have an $n$-prong singularity on $p$. Choose $k$ so that $h^k$ fixes the
boundary components and punctures of $S'$. Extend $h^k$ on $S$ by the identity outside $S'$.

We claim that if $h$ fixes a loop or ray, then that loop or ray is contained
in $\phi(S')$: because $h$ fixes the boundaries of $\phi(S)$,
 if $\gamma$ is a fixed ray or loop in $S$ not contained in $S'$,
then since $h$ fixes the boundaries of $\phi(S')$, by the definition of
$\phi^\leftarrow$, we have $\phi^\leftarrow(\gamma)$ is a loop in 
$\bar\L(S';p)$ which is fixed by $h$.  This is not possible, so we conclude
that any $\gamma$ fixed by $h$ has to be contained in $\phi(S')$
(and is thus a high-filling ray in $S'$).

We already know what the fixed points of $h$ are as high-filling rays:
they are the two $n$-prong foliations of the pseudo-Anosov $h$.
That the action of $h$ commutes with the map $\phi^\leftarrow$ similarly
implies that the limit under powers of $h$ of any ray or loop in
$\bar\L(S;p)$ must be one of these rays.  Thus $h$ is loxodromic and
has weight $n$.
\end{proof}

Theorem~\ref{theo:weights} is possible because we can embed subsurfaces
of arbitrary complexity into an infinite type surface.
The lemma below records the fact that homeomorphisms
on finite type surfaces have bounded weights, which is a corollary
of the Euler-Poincar\'{e}-Hopf formula.

\begin{lemma}\label{lemma:finite_type}
Let $S$ be a finite type surface with an isolated marked puncture $p$.
There is a uniform bound on $w(h)$ for $h \in \Mod(S;p)$.
\end{lemma}
\begin{proof}
As discussed previously, the cliques of a pseudo-Anosov element
in a finite type surface are just the leaves of the singular foliation
emanating from the marked puncture.
The Euler-Poincar\'{e}-Hopf formula states that
$2-2g = \sum_i (1-d(s_i)/2)$,
where the $s_i$ are the singular points of the foliation,
and $d(s_i)$ is the number of prongs.  A single
prong (which is the only positive contribution to the
right hand side) can only occur at a puncture, so because the number
of punctures is bounded, there is a uniform bound on 
the maximum $d(s_i)$ which can occur.  That is, there
is a uniform bound on the weight of a loxodromic element in $\Mod(S;p)$.
See~\cite{band_boyland}, Section 2.3.
\end{proof}

\begin{remark}
The loxodromic element $h$ defined in \cite{Juliette}, Section 4.1,
gives an example of a loxodromic element of weight $1$ which does not
preserve any finite type subsurface. It would be interesting to construct
examples of loxodromic elements of weight greater than $1$ which do not
preserve any finite type subsurface (up to isotopy).  It is not
difficult to find candidate homeomorphisms, but proving exactly
what the weight is remains difficult.
\end{remark}

\section{Application: construction of quasimorphisms}
\label{section:quasimorphisms}

\subsection{Main theorem and discussion}

\begin{theorem}\label{theo:quasimorphisms_weight} Let $S$ be a surface with an isolated marked puncture $p$.
Let $G$ be a subgroup of $\Mod(S;p)$ such that there exist two
loxodromic elements $g,h \in G$ with weights $w(g) \neq w(h)$.
Then the space $\tilde Q(G)$ of non-trivial quasimorphisms
on $G$ is infinite dimensional. 
\end{theorem}

We defer the proof of Theorem~\ref{theo:quasimorphisms_weight}
to the next section; we now discuss the motivation, as a series of corollaries.

\begin{corollary} Let $S$ be a surface with an isolated marked puncture $p$.
If $G$ is a subgroup of $\Mod(S;p)$ whose second
bounded cohomology group is trivial, then every loxodromic
element in $G$ has the same weight.
\end{corollary}

The following corollary is just a restatement of the hypotheses of
Theorem~\ref{theo:quasimorphisms_weight} which may be convenient for
computation.

\begin{corollary} \label{cor:group_actions} Let $S$ be a surface with an isolated marked puncture $p$.
Let $G$ be a subgroup of $\Mod(S;p)$ which contains two elements $g_1,g_2$.
Suppose that there are two essential finite type subsurfaces $S_1, S_2 \subseteq S$
based at $p$ such that
\begin{enumerate}
	\item $g_i$ preserves $S_i$ up to isotopy.
	\item The restriction of $g_i$ to $S_i$ is a pseudo-Anosov
	whose attractive and repulsive foliation have an $n_i$-pronged
	singularity at $p$.
	\item $n_1 \ne n_2$.
\end{enumerate}
Then the space $\tilde Q(G)$ of non-trivial quasimorphisms on $G$ is infinite dimensional. 
\end{corollary}

The combination of Theorems~\ref{theo:weights} and~\ref{theo:quasimorphisms_weight}
implies that this is not a vacuous subject:
\begin{corollary} \label{cor:quasim_on_MCG} Let $S$ be a surface with an isolated marked puncture $p$.
The space $\tilde Q(\Mod(S;p))$ of non-trivial quasimorphisms on
$\Mod(S;p)$ is infinite dimensional.  In particular, if
$G$ is any group which surjects onto $\Mod(S;p)$,
then $\tilde Q(G)$ is infinite dimensional.
\end{corollary}

Corollary~\ref{cor:quasim_on_MCG} can be used to ``lift''
our results about marked surfaces to more general situations.
For example, if $S$ is any surface with a puncture $p$ not accumulated
by genus, and $S'$ is obtained by forgetting enough punctures so that $p$
is isolated, then $\mathrm{PMCG}(S)$, the pure mapping class group
of $S$, surjects onto $\mathrm{PMCG}(S';p)$.  Since
$\mathrm{PMCG}(S';p)$ has an infinite-dimensional space of
quasimorphisms, $\mathrm{PMCG}(S)$ does too. Hence we have the following result.

\begin{corollary} \label{cor:quasim_on_pure_MCG}
Let $S$ be surface with at least one end which is not accumulated by genus. The space $\tilde Q(\PMod(S))$ of non-trivial quasimorphisms on
the pure mapping class group of $S$ is infinite dimensional. 
\end{corollary}

\subsection{Proof of Theorem~\ref{theo:quasimorphisms_weight}}

In this section, we prove Theorem~\ref{theo:quasimorphisms_weight} via a
series of lemmas. We will use Theorem 1.1 of \cite{Bestvina-Fujiwara}:
because $g$ and $h$ have disjoint cliques, they are independent.
It remains to prove that the quasi-axes of $g$ and $h$ are
\emph{anti-aligned}, i.e. $g \nsim h$ in the notation of \cite{Bestvina-Fujiwara}.
Let $\bar x=(x_n)_{n\in \Z}$ be a $(\kappa,\epsilon)$-quasi-axis of $g$ and
let $\bar y=(y_n)_{n\in \Z}$ be a $(\kappa,\epsilon)$-quasi-axis of $h$. We denote
by $B$ the Morse-constant associated to $(\kappa, \epsilon, \delta)$, where
$\delta$ is given by the hyperbolicity of the loop graph. Assume that $w(g) > w(h)$.
We will find a sub-segment $\alpha$ of $(x_n)$ such that for every $\phi \in \Mod(S;p)$,
we have that $\phi(\alpha)$ is not included in the $B$-neighborhood of
$\bar y$.  By definition, the existence of such a sub-segment $\alpha$ implies that $g$ and $h$ are anti-aligned; see~\cite{Bestvina-Fujiwara}.

We start by generalizing the map $A$ from \cite{Juliette}, with the
following definition:
\begin{definition} Let $f \in \Mod(S;p)$ be a loxodromic element,
and let $\CC^+(f)$ be its positive clique. We define
$A_f: \L(S;p) \to \N$ by $A_h(\ell) = k$ where $k$ is maximal such
that $\ell$ $k$-begins like an element of $\CC^+(h)$.
\end{definition}

We will need the following lemmas.

\begin{lemma}
\label{lemma:control of neighborhood-Af version}
Let $h \in \Mod(S;p)$ be loxodromic. For every $k,m \in \N$,
there exists $N \in \N$ such that for every $x,y \in \L(S;p)$ such
that $A_h(x) \geq N$ and $d(x,y)\leq m$ we have $A_h(y)\geq k$.
\end{lemma}

In other words, if two loops are close, and one of them $N$-begins
like a ray in $\CC^+(h)$, then the other is forced to $k$-begin
like a ray in $\CC^+(h)$.

\begin{proof}[Proof of Lemma~\ref{lemma:control of neighborhood-Af version}]
According to Lemma \ref{lemma:control_of_neighborhoods_of_geodesics_general_m},
for every $\lambda \in \CC^+(h)$, there exists $N_\lambda$ such that
if $x$ $N_\lambda$-begins like $\lambda$, then for every ray $y$ such
that $d(x,y)\leq m$, we have $A_h(y)\geq k$.
By Theorem~\ref{theorem:finite cliques}, $\CC^+(h)$ is finite, so we can
define $N:=\max \{N_\lambda | \lambda \in \CC^+(h)\}$. This $N$ satisfies
the lemma.
\end{proof}

\begin{lemma}
\label{lemma:Af is bounded on bounded sets}
Let $E$ be a bounded subset of $\L(S;p)$. For every
$h \in \Mod(S;p)$ with $h$ loxodromic, we have that 
$A_h(E)$ is bounded.
\end{lemma}
\begin{proof}
Let $z \in E$ and let $k=A_h(z)+1$. Denote by $m$ the diameter
of $E$. According to Lemma~\ref{lemma:control of neighborhood-Af version},
there exists $N$ such that for every $x,y \in L(S;p)$
such that $A_f(x) \geq N$ and $d(x,y)\leq m$, we have $A_f(y)\geq k$.
By contradiction, assume that for every $M \in \N$, there
exists $x \in E$ with $A_h(x) \geq M$. Choose $x \in E$
such that $A_h(x) \geq N$. For every $y \in E$, we have $d(x,y) \leq m$, thus $A_h(y) \geq k$ for every $y \in E$. This is a contradiction, because $z \in E$ and $A_h(z)=k-1$.
\end{proof}

\begin{lemma}\label{lemma:quasi-axis, Af}
Let $h \in \Mod(S;p)$ be a loxodromic element.
Let $(z_n)_{n \in \Z}$ be a $(\kappa,\epsilon)$-quasi-axis of $h$.
For every $k \in \N$, there exists $N \in \N$ such that for every
$n \geq N$, we have $A_h(z_n)\geq k$.
\end{lemma}

\begin{proof}
We will state a series of facts and then chain them together
to prove the lemma.  There is unfortunately a plethora of dummy variables.

Choose $\lambda \in \CC^+(h)$. Denote by $(a_n)_{n\geq 0}$ the
unicorn path $P(z_0,\lambda)$. 

First, by the unicorn path
construction, we have $\lim_{n \to \infty} A_h(a_n)=\infty$, 
so for every $j \in \N$, there exists $M$ such that for every
$m\geq M$, $A_h(a_m)\geq j$. 

Next, because $\lambda$ is high-filling,
$\lim_{n \to \infty} d(z_0,a_n)=\infty$.  Thus, for all $N,k \in \N$,
there is $M\in \N$ such that if we have a point $x \in L(S;p)$
with $d(z_0,x) > N$ and $d(x,a_m) < k$, then $m > M$.  That is,
if $x$ is very far from $z_0$ but close to some $a_m$, then $m$
must be large.

Next, the unicorn path $(a_n)_n = P(z_0,\lambda)$ is uniformly close to
a geodesic sequence to $\lambda$, and the quasi-axis of $h$ is also
uniformly close to this geodesic.  Hence, the unicorn path
$P(z_0,\lambda)$ is uniformly close to the quasi-axis $(z_n)_n$.
That is, there is some $L$ such that for all $n$, there is
some $m$ so that $d(z_n,a_m) < L$.

Finally, let us be given $k \in \N$ as in the lemma.  By
Lemma~\ref{lemma:control of neighborhood-Af version}, there is
some $M$ such that if $d(x,y) < L$ and $A_h(x) > M$, then
$A_h(y) > k$.  By our first observation, there is some $M'$
such that for all $m > M'$, we have $A_h(a_m) > M$.
By our second observation, there is some $N'$ such that
if $d(z_0, x) > N'$ and $d(x, a_m) < L$, then $m > M'$.
Because $(z_n)_n$ is a quasi-axis, there is some $N$ such that
if $n > N$, then $d(z_0,z_n) > N'$.  Chaining
these together, we conclude that if $n > N$, we have $d(z_0, z_n) > N'$;
by our third observation there is some $a_m$ so that $d(z_n, a_m) < L$,
which implies that $m>M'$, which implies that $A_h(a_m) > M$,
and again applying that $d(z_n, a_m) < L$, we have $A_h(z_n) > k$.
\end{proof}

\begin{lemma} \label{lemma:quasi-axis nbhd}
Let $h \in \Mod(S;p)$ be a loxodromic, and let $(z_n)$ be a
quasi-axis of $h$. For every $k, C \in \N$, there exists
$M \in \N$ such that for every $m \geq M$, we have 
$A_h(N_C(z_m))\geq k$, where $N_C(z_m)$ denotes the
$C$-neighborhood of $z_m$.
\end{lemma}
\begin{proof}
Let $k, C \in \N$. By Lemma~\ref{lemma:control of neighborhood-Af version},
there exists $k'\in \N$ such that for every $x \in L(S;p)$
with $A_h(x)\geq k'$, we have $A_h(N_C(x)) \geq k$.
By Lemma~\ref{lemma:quasi-axis, Af}, there exists $M \in \N$ such
that for every $m\geq M$, we have $A_h(z_m) \geq k'$, hence
$A_h(N_C(z_m))\geq k$.
\end{proof}

We now restate Lemma~\ref{lemma:quasi-axis nbhd} in the language
of the conical cover.  Here we think of a high-filling
ray $\lambda$ as a point on the $S^1$ boundary of the conical
cover, and we can consider neighborhoods of loops and rays
as open intervals in $S^1$.

\begin{lemma}\label{lemma:quasi-axis in neighborhood of cliques}
Let $h \in \Mod(S;p)$ be loxodromic. Let $(z_n)_n$ be a quasi-axis
of $h$.  For every $C>0$, there exists $N\in \N$ and a
neighborhood $B(\lambda) \subset \S^1$ for each element
$\lambda \in \CC^+(h) \cup \CC^-(h)$ seen in the boundary of
the conical cover, such that:
\begin{itemize}
\item For any two distinct $\lambda$ and $\lambda'$ in $\CC^+ \cup \CC^-$, the neighborhoods $B(\lambda)$ and $B(\lambda')$ are disjoint.
\item The $C$-neighborhood of the half-quasi-axis $(z_n)_{n>N}$ is included in $\cup_{\lambda \in \CC^+} B(\lambda)$.
\item The $C$-neighborhood of the half-quasi-axis $(z_n)_{n<N}$ is included in $\cup_{\lambda \in \CC^-} B(\lambda)$.
\end{itemize}
\end{lemma}
\begin{proof}
A neighborhood $B(\lambda)$ of a high-filling ray $\lambda$
corresponds to the set of all rays and loops which $k$-begin
like $\lambda$ for some $k$ (see Section~\ref{section:cover_convergence}).
There are finitely many elements of
$\CC^+(h) \cup \CC^-(h)$, so we can take a $k$ so that if a ray or
loop $k$-begins like an element of $\CC^+(h) \cup \CC^-(h)$,
then it lies in $B(\lambda)$ for some $\lambda$, 
and the $B(\lambda)$ are disjoint.
Now we apply
Lemma~\ref{lemma:quasi-axis nbhd} for both $h$ and $h^{-1}$:
there is some $M$ such that
for all $m\ge M$ we have $A_h(N_C(z_m))\ge k$ and
for all $m\ge M$ we have $A_{h^{-1}}(N_C(z_{-m}))\ge k$, meaning that 
anything in the $C$-neighborhood of the quasi-axis in a
sufficiently positive direction must $k$-begin like
a ray in $\CC^+(h)$, and anything in a sufficiently
negative direction must $k$-begin like a ray in $\CC^-(h)$.
(Note the negative sign on the quasi-axis, because it gets reversed
for $h^{-1}$).  The $N$ desired in the lemma is then the $M$ we have
just found.
As we observed above, this implies that as points on the boundary
of the conical cover, we have the last two bullet points
of the lemma.
\end{proof}

\begin{lemma} \label{lemma:anti-aligned}
Let $g,h \in \Mod(S;p)$ be two loxodromic elements with
different weights $w(g) \ne w(h)$.
Then in the language of~\cite{Bestvina-Fujiwara},
$g$ and $h$ are independent and anti-aligned ($g\nsim h$).
\end{lemma}

\begin{proof}
Assume without loss of generality that $w(g)>w(h)$. Let
$C$ be the Morse constant of the loop graph. Let $(z_n)_n$
be a quasi-axis of $g$ constructed as follows.
By Theorem~\ref{theorem:finite cliques}, the elements
of the two cliques $\CC^+(g)$ and $\CC^-(g)$ alternate on the
boundary $S^1$ of the conical cover. They bound $2w(g)$
segments, on which the dynamics of $g$ is very simple. We
choose $w(g)$ loops $(z_1,\ldots,z_{w(g)})$ so that no two of
them are in adjacent segments of $S^1 - (\CC^+(g)\cup \CC^-(g))$.
Choose mutually disjoint segments in $\L(S;p)$ between $z_i$ and
$z_{i+1}$ for every $1\leq i \leq w(g)$, and a segment between
$z_{w(g)}$ and $g(z_{1)})$. Denote by $Z$ the union of these
$w(g)+1$ segments. We define the quasi-axis $(z_n)_n$ of $g$ as
the orbit of $Z$ by $g$.

Let $(z'_n)$ be a quasi-axis of $h$. Let $N$ and $N'$, be as in
Lemma~\ref{lemma:quasi-axis in neighborhood of cliques}
for $g$ and $h$, respectively.
Also define neighborhoods $B(\lambda)$ as in
Lemma~\ref{lemma:quasi-axis in neighborhood of cliques}
for both $g$ and $h$, for all $\lambda$ in $\CC^\pm(g) \cup \C^\pm(h)$.
Let $W=(w_1,\ldots,w_K)$ be a long sub-segment of $(z_n)_n$ which
contains covers the intervals from $z_{-N-N'-100C}$
through $z_{N+N'+100C}$.
By the construction of $(z_n)_n$, 
the elements at the beginning and end of this sub-segment
$W$ must lie in the neighborhoods $B(\lambda)$.  Specifically,
for $1\leq i,j \leq w(g)$ with $i\ne j$, we have $w_i \in B(\lambda_i)$
and $w_j \in B(\lambda_j)$, where
$B(\lambda_i) \cap B(\lambda_j) = \varnothing$ and
$\lambda_i, \lambda_j \in \CC^-(g)$.  We also have
$w_{K-i} \in B(\lambda_{K-i})$ and $w_{K-j} \in B(\lambda_{K-j})$,
where $B(\lambda_{K-i}) \cap B(\lambda_{K-j}) = \varnothing$
and $\lambda_{K-i}, \lambda_{K-j} \in \CC^-(g)$.

Now suppose towards a contradiction that we have
$\phi \in \Mod(S;p)$ which maps $W$ into the $C$-neighborhood
of $(z'_n)_n$.  The interval $\phi(W)$ might not be centered
about $z'_0$, but we can shift it back with some power of $h$.
That is, there exists $k$ such that $h^k\phi(z_0)$ is a uniformly
bounded distance from $z'_0$.  Then the beginning and
end of the image subsegment $h^k\phi(W)$ are far from $z'_0$,
and thus:
\begin{itemize}
\item $h^k\phi(w_i)\in \cup_{\lambda \in \CC^-(h)}B(\lambda)$ for every $1\leq i \leq w(g)$.
\item $h^k\phi(w_{K-i})\in \cup_{\lambda \in \CC^+(h)}B(\lambda)$ for every $1\leq i \leq w(g)$.
\end{itemize}
Moreover, $h^k \phi$ acts on the boundary $S^1$ of the conical cover by a homeomorphism.
So we now have $2w(g)$ points which must alternate around $S^1$, and
they are mapped via a circle homeomorphism into $2w(h) < 2w(g)$
alternating intervals, which is impossible.  We conclude
that there can be no such $\phi$, and thus $g$ and $h$ are
anti-aligned.
\end{proof}

\addtocontents{toc}{\protect\vskip15pt}
\appendix
\section{Geodesic representatives of proper arcs}

In Section~\ref{sec:surfaces}, we constructed a
decomposition of an arbitrary surface by cutting it along
a collection of disjoint proper simple arcs.  A key
fact in the proof of Theorem~\ref{thm:hyperbolic} is
that in an infinite type hyperbolic surface,
a simple proper arc is ambient isotopic to a geodesic.
This section is devoted to the proof of this fact.
Although experts agree that this should be straightforward,
we are not aware of a direct proof in the literature.

We call \emph{proper arc} the image of the real open segment
$(0,1)$ by any immersion $\phi$ such that for any compact
subsurface $S'$ of $S$, there exists some $\varepsilon>0$ so
that $\phi((0,\varepsilon)\cup (1-\varepsilon,1))$ is disjoint
from $S'$.

We will take a general view that an \emph{isotopy} of $S$
is any continuous family of homeomorphisms containing
the identity.  There are two particular kinds of isotopies
in which we will be interested.

First, we say that a proper arc $\alpha$ is \emph{ambient isotopic}
to another arc $\beta$ if there is an isotopy
$\Phi:[0,1]\times S \to S$ with $\Phi(0,\cdot)$ the identity
and $\Phi(1,\cdot)\circ \alpha = \beta$.  Note 
that for all $t$, $\Phi(t,\cdot) \circ \alpha$ is a proper arc.

Second, we say that a proper arc $\alpha$ is
\emph{non-essential} if there exists
an isotopy $\Phi:[0,1)\times S \to S$ such that
$\Phi(0,\cdot)$ is the identity and for every compact
subsurface $S'\subseteq S$ there is some $T$ such that the
image of the composition $\Phi(t,\cdot)\circ \alpha$,
thought of as a subset of $S$, is disjoint from $S'$ for all $t>T$.
In other words, $\alpha$ is non-essential if it can be pushed off
$S$ out to infinity.  Note that the maps $\Phi(t,\cdot)$ will 
necessarily be rather violent as $t$ tends to $1$.
The opposite of non-essential is essential.

We remark that we could instead focus on homotopies of
the proper arcs themselves rather than these ambient isotopies.
However, because our end goal is to say something
about the mapping class group of $S$, it is useful to
start with the language of isotopies.  We also
remark that difference in domain between $[0,1]$ and $[0,1)$
is quite important.  For example, given a proper arc
$\alpha$ in $S$, we can construct an isotopy
$\Phi:[0,1)\times S \to S$ such that the pointwise limit of
$\Phi(t,\cdot)\circ \alpha$ is contained in a compact ball (by
pulling the ends of $\alpha$ in from infinity).  However,
this does not imply that proper arcs are ambient isotopic
to finite compact arcs, because there is no way to continuously
extend $\Phi$ to $\{1\}\times S$.

The goal of this section is to prove the following Lemma:

\begin{lemma}\label{lem:isotopy_of_arcs}
Let $S$ be an infinite type orientable surface without boundary, equipped with a complete hyperbolic metric of the first kind. Let $\alpha$ be a proper simple arc in $S$ which is essential. Then $\alpha$ is isotopic (by an ambient isotopy) to a unique geodesic arc.
\end{lemma}

The proof of this Lemma follows from the following Lemmas.

\begin{lemma}\label{lem:proper_alpha_transverse}
Let $S$ and $\alpha$ be as in Lemma~\ref{lem:isotopy_of_arcs}.
Let $\Gamma$ be a locally finite collection of simple closed geodesic
arcs in $S$.  Then $\alpha$ is ambient-isotopic to a proper arc
which is a smooth immersion transverse to $\Gamma$.
\end{lemma}
\begin{proof}
This requires an arbitrarily small perturbation of $\alpha$.
\end{proof}

\begin{lemma}\label{lem:proper_alpha_locally_finite}
Let $S$ and $\alpha$ be as in Lemma~\ref{lem:isotopy_of_arcs}.
Let $\Gamma$ be a locally finite collection of simple closed geodesic
curves in $S$.  Then $\alpha$ is ambient-isotopic to a proper arc
which intersects each element of $\Gamma$ finitely many times.
\end{lemma}
\begin{proof}
This is immediate from Lemma~\ref{lem:proper_alpha_transverse} because
a proper arc which is transverse to a compact curve can only intersect it
finitely many times.
\end{proof}

Given a smooth arc $\alpha$ and a collection $\Gamma$ of simple closed curves in $S$,
a bigon between them is an immersion of a bigon into $S$ whose boundary maps
to two arcs, one in $\alpha$ and one in some element $\gamma$ of $\Gamma$.
We say that $\alpha$ is in minimal position with respect to $\Gamma$ if there
is no bigon between them.  Note that the image of the boundary of any
bigon produces two intervals, one in $\alpha$ and one in some $\gamma \in \Gamma$.
For any point $p \in \alpha$, we can ask how many bigon-intervals contain $p$.

\begin{lemma}\label{lem:proper_alpha_finite_bigons}
Let $S$ and $\alpha$ be as in Lemma~\ref{lem:isotopy_of_arcs}.
Let $\Gamma$ be a locally finite collection of simple closed geodesic
curves in $S$.  Then $\alpha$ is ambient isotopic to an arc $\alpha'$
such that for any point $p \in \alpha'$, there are only
finitely many bigon-intervals containing $p$.
\end{lemma}
\begin{proof}
Consider the set of intersections between $\alpha$ and $\Gamma$.  Because $\Gamma$
is geodesic and $\alpha$ is proper, after acting by an ambient isotopy by
Lemma~\ref{lem:proper_alpha_locally_finite},
we may assume this set of intersections is closed, discrete, and discrete as a 
subset of $\alpha$.  A bigon interval is specified by the pair of intersection
points between $\alpha$ and $\Gamma$ which it runs between.  Note that
this pair of intersection points must lie on a single $\gamma \in \Gamma$.
Each $\gamma \in \Gamma$ can intersect $\alpha$ only finitely many times by
Lemma~\ref{lem:proper_alpha_locally_finite}.  Thus, no intersection between
$\alpha$ and $\gamma$ can serve as an endpoint of infinitely many
bigon-intervals.  This implies that if $p \in \alpha$ is contained in
infinitely many bigon-intervals, then there is a sequence of nested bigons
whose union contains all of $\alpha$.  This produces a null-homotopy, 
which contradicts the assumption that $\alpha$ is not null-homotopic.
\end{proof}

\begin{lemma}\label{lem:minimal_isotopy}
Let $S$ and $\alpha$ be as in Lemma~\ref{lem:isotopy_of_arcs}.
Let $\Gamma$ be a locally finite collection of simple closed geodesic
curves in $S$.  There is an ambient isotopy of $S$ which puts $\alpha$
into minimal position with respect to $\Gamma$.
\end{lemma}
\begin{proof}
Choose an arbitrary point $p$ on $\alpha$.  It is contained in finitely
many bigon intervals by Lemma~\ref{lem:proper_alpha_finite_bigons}
(if this number is zero, the next steps are vacuous).  Choose any one of
these intervals.  It corresponds to some bigon $B$.  The $\alpha$ boundary
of $B$ may intersect $\Gamma$, and the $\Gamma$ boundary
may intersect $\alpha$.  However, note these intersections must be homotopically
trivial (because $B$ is a disk).  By Lemma~\ref{lem:proper_alpha_locally_finite},
the boundary of $B$ can contain only finitely many of these intersections.
The interior of $B$ may contain yet more intersections, but by the properness of $\alpha$
there are only finitely many.  All of these (necessarily trivial) intersections
are themselves associated with bigon intervals.  By doing finitely many ambient
isotopies in a tubular neighborhood of $B$, we can undo all these interior bigons and then finally undo $B$.
Note that a point $p' \in \alpha$ can be affected by these isotopies only
if it is contained in a bigon-interval, and each isotopy which affects it
removes a bigon-interval from $p'$.

All this removed (at least) a single bigon-interval from $p$.  We repeat this process
until $p$ is contained in no bigon-intervals.  Next, there is some
closest bigon-interval to $p$.  Choose a point in its interior, and 
repeat indefinitely, proceeding outward from $p$.
This produces an infinite product of ambient isotopies.
We claim this infinite product in fact produces an honest isotopy.  To see
this, consider Lemma~\ref{lem:proper_alpha_finite_bigons}.  Every single
ambient isotopy we apply removes a bigon-interval from $\alpha$, so for
every point in $\alpha$, there is some finite time after which it is not contained
in any bigon-intervals and is thus untouched by further isotopies.
Because $\alpha$ is proper, for any compact set $K \subseteq S$, we know that
$\alpha$ must exit $K$, so there is some finite time after which no ambient isotopy
affects $K$.  That is, this infinite sequence converges (is eventually constant)
on compact sets, and thus gives an isotopy of $S$.
\end{proof}

\begin{proof}[Proof of Lemma~\ref{lem:isotopy_of_arcs}]
Apply Lemma~\ref{lem:minimal_isotopy} to put $\alpha$ into minimal
position with respect to the cuffs of the pants decomposition.
Because $\alpha$ is proper and non null-homotopic, a lift of it has
two endpoints on the boundary of a universal cover, and there is a 
unique geodesic $\tilde\beta$ between those endpoints.  Then $\tilde\beta$
projects to a geodesic $\beta$ on $S$.  Because $\alpha$ is in minimal position
and has the same endpoints as $\beta$, they must cross the same cuffs
in the same order.  We can apply local ambient isotopies to simultaneously
make all intersections between $\alpha$ and the cuffs identical to the
intersections between $\beta$ and the cuffs.  Then, we apply
Theorem 3.1 of Epstein~\cite{Epstein} within each pair of pants
simultaneously to make $\alpha$ and $\beta$ coincide.
\end{proof}


\begin{thebibliography}{alpha}


\bibitem{Aramayona-F-P}
	Javier Aramayona, Ariadna Fossas, and Hugo Parlier
	\emph{Arc and curve graphs for infinite-type surfaces}, 
	Proc. Amer. Math. Soc. 145 (2017), no. 11, 4995--5006

\bibitem{Aramayona-V}
Javier Aramayona and Ferr\'an Valdez,
\emph{On the geometry of graphs associated to infinite-type surfaces}, arXiv:1605.05600

\bibitem{band_boyland}
Gavin Band and Philip Boyland
\emph{The Burau estimate for the entropy of a braid}
Algebr. Geom. Topol.
Volume 7, Number 3 (2007), 1345-1378.


\bibitem{Juliette}
	Juliette Bavard, 
	\emph{Hyperbolicit\'{e} du graphe des rayons et quasi-morphismes sur un gros groupe modulaire}, 
	Geom. Topol. {\bf 20}, no. 1, (2016) 491--535
	
\bibitem{Juliette-Anthony}
	Juliette Bavard and Anthony Genevois, 
	\emph{Big mapping class groups are not acylindrically hyperbolic}, 
	Math. Slovaca {68} (2018), no. 1, 71--76.
	
	
\bibitem{boundary}
	Juliette Bavard and Alden Walker,
	\emph{The Gromov boundary of the ray graph}, arXiv:1608.04475, to appear in Transactions of the AMS.
	
	
\bibitem{Bestvina-Fujiwara}	
	Mladen Bestvina and Koji Fujiwara
	\emph{Bounded cohomology of subgroups of mapping class groups}, Geom. Topol. {\bf 6} (2002), 69--89. 


\bibitem{Calegari-blog}
	Danny Calegari,
	\emph{Big mapping class groups and dynamics},\\
	\texttt{https://lamington.wordpress.com/2009/06/22/big-mapping-class-groups-and-dynamics/}

\bibitem{Calegari-circular}
	Danny Calegari, \emph{Circular groups, planar groups, and the Euler class} Geom. Topol. Mon. 7 (2004), 431-491
	
\bibitem{Durham-F-V}
	Matthew Gentry Durham, Federica Fanoni, and Nicholas G. Vlamis
	\emph{Graphs of curves on infinite-type surfaces with mapping class group actions}, arXiv:1611.00841
	
\bibitem{Epstein}
       D. B. A. Epstein, 
  \emph{Curves on $2$-manifolds and isotopies},
{Acta Math.}, {\bf 115}, (1966), {83--107}.


\bibitem{Hensel-Przytycki-Webb}
	Sebastian Hensel, Piotr Przytycki and Richard Webb,
	\emph{1-slim triangles and uniform hyperbolicity for arc graphs and curve graphs}, J. Eur. Math. Soc. (JEMS), {\bf{17}} (2015) 755--762

\bibitem{Masur-Minsky}
	Howard A. Masur and Yair N. Minsky,
	\emph{Geometry of the complex of curves. II. Hierarchical structure},
	Geom. Funct. Anal., 10(4):90--974  (2000).


\bibitem{Pho-On}
	Witsarut Pho-on, \emph{Infinite Unicorn Paths and Gromov Boundaries}, 
	Groups Geom. Dyn. 11 (2017), no. 1, 353--370.

\bibitem{Rasmussen}
Alexander J. Rasmussen,
\emph{Uniform hyperbolicity of the graphs of nonseparating curves via bicorn curves}, arXiv:1707.08283

\bibitem{Richards}
	I. Richards, 
	\emph{On the classification of noncompact surfaces}, 
	Trans. Amer. Math. Soc. {\bf 106} (1963) 259--269

\bibitem{Schleimer}
	Saul Schleimer,
	\emph{Notes on the curve complex}.\\
	\texttt{http://homepages.warwick.ac.uk/~masgar/Maths/notes.pdf}	
	

\end{thebibliography}
\end{document}